\documentclass[english]{amsart} 

\usepackage[normalem]{ulem}
\usepackage{enumerate}
\usepackage{amssymb}
\usepackage{amsbsy}
\usepackage{color}
\usepackage{comment}
\usepackage{array}
\usepackage{enumerate}
\usepackage{mathrsfs}
\usepackage{mathtools}
\usepackage{multicol}

\usepackage[curve,all]{xy}
\xyoption{matrix}
 \xyoption{curve}
 \xyoption{color}
 \xyoption{line}
\xyoption{arc}

\usepackage[shortlabels]{enumitem}
\setlist[enumerate]{label=\rm{(\arabic*)}}
\setlist[enumerate,2]{label=\rm({\it\roman*})}
\setlist[itemize]{label=\raisebox{0.25ex}{\tiny$\bullet$}}
\usepackage[backref, colorlinks, linktocpage, citecolor = blue, linkcolor = blue]{hyperref}

\theoremstyle{plain}
\newtheorem{theorem}{Theorem}[section]
\newtheorem{lemma}[theorem]{Lemma}
\newtheorem{proposition}[theorem]{Proposition}

\newtheorem{corollary}[theorem]{Corollary}
\newtheorem{example}[theorem]{Example}

\newtheorem{question}[theorem]{Question}

\newtheorem*{maintheorem}{Theorem}
\newtheorem*{mainProp}{Proposition}

\theoremstyle{definition}
\newtheorem{definition}[theorem]{Definition}
\newtheorem{notation}[theorem]{Notation}

\newtheorem{remark}[theorem]{Remark}

\newtheorem*{acknowledgement}{Acknowledgement}

\theoremstyle{remark}

\newcommand{\incl}[1][r]{\ar@<-0.2pc>@{^(-}[#1] \ar@<+0.2pc>@{-}[#1]}

\renewcommand{\P}{\mathbb{P}}

\newcommand{\FF}{\mathcal{F}}

\DeclareMathOperator{\bir}{bir}
\DeclareMathOperator{\lin}{line}
\DeclareMathOperator{\Amp}{Amp}
\DeclareMathOperator{\Nef}{Nef}
\DeclareMathOperator{\Eff}{Eff}

\DeclareMathOperator{\pf}{pf}

\newcommand{\Spec}{\mathrm{Spec}}

\renewcommand{\SS}{\mathcal{S}}

\newcommand{\p}{\mathbb{P}}

\newcommand{\pr}{\mathrm{pr}}
\newcommand{\Q}{\mathcal{Q}}

\newcommand{\C}{\mathbb{C}}

\newcommand{\M}{\overline{M}}
\newcommand{\MmXbeta}{\overline{M}_{0,m}(X,\beta)}
\newcommand{\MzeroXbeta}{\overline{M}_{0,0}(X,\beta)}
\newcommand{\MtwoXbeta}{\overline{M}_{0,2}(X,\beta)}
\newcommand{\MmXd}{\overline{M}_{0,m}(X,d)}

\newcommand{\MoneXd}{\overline{M}_{0,1}(X,d)}
\newcommand{\MtwoXd}{\overline{M}_{0,2}(X,d)}
\newcommand{\MbirzeroXbeta}{\overline{M}^{\bir}_{0,0}(X,\beta)}

\newcommand{\HH}{\mathcal{H}}

\newcommand{\U}{\mathcal{U}}
\newcommand{\Z}{\mathbb{Z}}

\newcommand{\G}{\mathbb{G}}

\renewcommand{\O}{\mathcal{O}}

\newcommand{\Gr}{\mathrm{Gr}}

\DeclareMathOperator{\SL}{SL}
\DeclareMathOperator{\Aut}{Aut}

\DeclareMathOperator{\rk}{rk}

\DeclareMathOperator{\Hilb}{Hilb}

\DeclareMathOperator{\Pic}{Pic}

\DeclareMathOperator{\Ker}{Ker}

\DeclareMathOperator{\sgn}{sgn}

\DeclareMathOperator{\ev}{ev}
\DeclareMathOperator{\sm}{sm}
\DeclareMathOperator{\Mor}{Mor}
\DeclareMathOperator{\odd}{odd}
\DeclareMathOperator{\even}{even}
\DeclareMathOperator{\Sym}{Sym}
\DeclareMathOperator{\Spin}{Spin}
\DeclareMathOperator{\Sp}{Sp}

\title[Rational curves on $V_5$]{Rational curves on $V_5$ and rational simple connectedness}

\date{\today}

\author[Andrea Fanelli]{Andrea Fanelli}
\address{Laboratoire de Math\'ematiques de Versailles, UVSQ, CNRS, Universit\'e Paris-Saclay, 78035 Versailles, France.}

\email{andrea.fanelli@uvsq.fr}

\author[Laurent Gruson]{Laurent Gruson}
\address{Laboratoire de Math\'ematiques de Versailles, UVSQ, CNRS, Universit\'e Paris-Saclay, 78035 Versailles, France.}
\email{laurent.gruson@uvsq.fr}

\author{Nicolas Perrin}
\address{Laboratoire de Math\'ematiques de Versailles, UVSQ, CNRS, Universit\'e Paris-Saclay, 78035 Versailles, France.}
\email{nicolas.perrin@uvsq.fr}


\begin{document}

\begin{abstract}
In this paper the notion of rational simple connectedness for the quintic Fano threefold $V_5\subset \mathbb{P}^6$ is studied and unirationality of the moduli spaces $\overline{M}_{0,0}^{\text{bir}}(V_5,d)$, with $d \ge 1$ is proved.
Many further unirationality results for special moduli spaces of rational curves on quadric hypersurfaces and del Pezzo surfaces are obtained via explicit birational methods. 
\end{abstract}

\maketitle
\tableofcontents


\section*{Introduction}\label{Sec:Intro}

Given a smooth polarised rationally connected variety $X$ over $\C$, we are interested in studying the Kontsevich moduli spaces $\M_{0,m}(X, d)$ of rational curves on $X$ with sufficiently large degree $d$.
Our main motivation comes from the series of works by de Jong and Starr (cf.~\cite{dJS_2006a}, \cite{dJS_2006b}, \cite{dJS_2006c}, \cite{dJS_2007}, etc.). Starting form the seminal work \cite{GHS_2003}, whose main result guarantees that a rationally connected fibration over a smooth curve has a section, de Jong and Starr explore numerical and geometric conditions on the general fibre of a rationally connected fibration over a surface which guarantee the existence of a rational section.

This deep analysis originated the new notion of \emph{rational simple connectedness} and was applied to the case of homogeneous varieties to prove \emph{Serre's Conjecture II} (cf.~\cite{dJHS_2011}).
This notion has several variations, but heuristically is the algebraic analogue of simple connectedness in topology: a variety $X$ is rationally simply connected if for $d$ sufficiently large, the evaluation map
$$\ev_2\colon \M_{0,2}(X,d)\to X^2$$
is dominant and the general fibre is rationally connected (cf.~Definition~\ref{Def:RSC}). 
\\The reason why one requires this property for large enough $d$ is that the behaviour of moduli spaces of rational curves in low degree can be atypical or, more simply, these spaces can be empty.

In \cite{dJS_2006c}, the authors also introduce a strong version of rational simple connectedness, which requires the existence of a ``very free'' ruled surface in $X$ (a so called \emph{very twisting scroll}). In this paper we start from the easier notion introduced in \cite[Section~1]{dJS_2006c} for varieties with Picard-rank one and propose a general definition for arbitrary rationally connected varieties (cf.~Definition~\ref{Def:RSC} of \emph{canonically rationally simply connected variety}).

Rational simple connectedness is very subtle and very few examples are known; the picture has been clarified for complete intersections in projective spaces (\cite{dJS_2006c}, \cite{D_2015}), homogeneous spaces (\cite{dJHS_2011}, \cite{BCMP_2013}), and hyperplane sections of Grassmannians (\cite{F_2010}). The main result of this paper is the following.

\begin{maintheorem}
{\rm (= Theorem~\ref{Main_V_5})}
The Fano threefold $V_5\subset \P^6$, obtained as linear section of $\Gr(2,5)\subset \P^9$, is rationally simply connected.
\end{maintheorem}

The importance of this example comes from its very specific geometry: it is a smooth rational quasi-homogeneous (with respect to a $\SL_2$-action) Fano threefold which is not \emph{2-Fano} (cf.~\cite{AC_2012}, \cite{AC_2013}).
Although rational simple connectedness in \emph{not} a birational property, the birational geometry of $V_5\subset \P^6$ is well known (cf.~Section~\ref{Sec:V5_1}) and our strategy consists in reducing the study of moduli spaces of rational curves on $V_5$ to some special moduli spaces on the quadric threefold $\Q_3$. The same methods give the following.

\begin{mainProp}
{\rm (= Proposition~\ref{Prop:V5RatComp})}
For any $d \geq 1$ the moduli space $\M_{0,0}^{\bir}(V_5,d)$ of stable maps birational onto the image is irreducible, unirational of dimension $2d$.
\end{mainProp}
Our methods also provide some results for del Pezzo surfaces.
\begin{maintheorem}
{\rm (= Theorem~\ref{Thm:dPRSC})}
Let $X_\delta$ be a del Pezzo surface of degree $\delta\ge 5$. Then $X_\delta$ is canonically (strongly) rationally simply connected.
\end{maintheorem}
This paper is organised as follows. In Section~\ref{Sec:RSC} the definition of \emph{canonically rationally simply connected} variety and its strong variant are introduced and some general results on moduli spaces of rational curves on smooth varieties are recalled. 
\\Section \ref{Sec:First_Examples} is devoted to first examples: the projective space $\P^n$ and the quadric hypersurface $\Q_n\subset \P^{n+1}$. We obtain precise statements on moduli spaces of rational curves on these varieties which will be crucial for our birational approach to rational simple connectedness. 
\\In Section~\ref{Sec:delPezzo} we test our new notion on del Pezzo surfaces (de Jong-Starr's definition only applies to varieties with Picard rank one) and the main result is proved in the first part of Section~\ref{Sec:V5_1}. We conclude studying rationality of moduli spaces of rational curves on $V_5$ in low degree: new proofs of rationality of moduli spaces of quintic and sextic rational curves are provided, via explicit birational geometric methods.

\begin{acknowledgement}
We like to thank Sasha Kuznetsov, Andrea Petracci and Francesco Zucconi for interesting comments related to this work.
The first-named author is currently funded by the Fondation Math\'ematique Jacques Hadamard.
\end{acknowledgement}


\section{A notion of rational simple connectedness for Fano varieties}\label{Sec:RSC}

Let $X$ be a smooth projective variety of dimension $n$ over $\C$.
We follow \cite{FP_1997} for the notation of Kontsevich moduli spaces.

\begin{definition}\label{Def:KModuli}
Let $\MmXbeta$ denote the (coarse) moduli space of stable $m$-pointed rational curves on $X$ of degree $\beta \in H_2(X,\Z)$. Let $\ev=\ev_{\MmXbeta}$ denote
the \emph{evaluation morphism}
\begin{equation}\label{evaluation}
\ev\colon \MmXbeta \to X^m, 
\ \ \ 
(f; \  p_1,\ldots, p_n) \mapsto (f(p_1), \ldots, f(p_n)).
\end{equation}
\end{definition}

We recall here some well known facts about these moduli spaces for the reader's convenience.
\begin{itemize}
\item The moduli spaces $\MmXbeta$ exist as (possibly not connected) projective varieties (cf. \cite[Theorem~1]{FP_1997});
\item the expected dimension of $\MmXbeta$ is
\begin{equation}\label{expected_dimension}
\dim(X)-(K_X\cdot \beta) + m - 3;
\end{equation}
\item if $X$ is convex (i.e.~if for every morphism $\P^1 \to X$, one has $H^1(\P^1, T_{X|\P^1})=0$), 
then the spaces $\MmXbeta$ are normal of pure expected dimension (cf.~\cite[Theorem~2]{FP_1997}).
\emph{Rational homogeneous varieties} are the standard examples of convex varieties and on those ones the moduli spaces $\MmXbeta$ behaves particularly well 
(cf.~\cite[Sections~7--10]{FP_1997}, \cite{T_1998}, \cite{KP_2001}, \cite{P_2002}).
\end{itemize}

\begin{definition}\label{Def:Md}
Let $(X,H)$ be a polarised smooth variety. Then, for any $m\ge 0$, the degree $d$ moduli space of stable $m$-pointed rational curves on $X$ is defined as
$$\MmXd:= \bigsqcup_{(H\cdot \beta) = d}\MmXbeta.$$
\end{definition}

A smooth variety $X$ is \emph{uniruled} (resp. \emph{rationally connected}) if any general point in $X$ is contained in a rational curve (resp. if any two general points in $X$ can be connected by a rational curve). One can refine these definitions for polarised varieties.

\begin{definition}\label{Def:RC}
Let $(X,H)$ be a smooth polarised variety. Then $(X,H)$ is $d$-uniruled (resp. $d$-rationally connected) for some integer $d>0$ if the evaluation morphism
\begin{equation}\label{evaluation_2}
\ev_1\colon \MoneXd \to X \hspace{.5cm}\mbox{ (resp. }\ev_2\colon \MtwoXd \to X\times X\mbox{)}
\end{equation}
is dominant. The minimal degree such that the evaluation map \eqref{evaluation_2} is dominant is called \emph{minimal rational covering degree} (resp.~\emph{minimal rational $2$-connecting degree}) and is denoted by $d_{(X,H)}(1)$ (resp.~$d_{(X,H)}(2)$).
\end{definition}

\begin{remark}\label{Rem:VMRT}
Even if this is essential for our analysis, we recall the existence of another moduli space of rational curves which parametrises the tangent directions to rational curves of minimal degree that pass through a general point $x\in X$: the \emph{variety of minimal rational tangents} or \emph{VMRT}. For the rigorous definition and fundamental results see the survey \cite{KS_2006}.
\end{remark}

It is well known that uniruled varieties are in general \emph{not} rationally connected (e.g.~ruled surfaces over high-genus curves). Nonetheless, for rationally connected varieties, one can connect more than two points with rational curves, as the following result shows.
\begin{proposition}\label{Prop:dm}
Let $(X,H)$ be smooth rationally connected polarised variety. Then there exists an integer $d>0$ such that the evaluation map
\begin{equation}\label{evaluation_m}
\ev_m\colon \MmXd \to X^m
\end{equation}
is dominant. The minimal degree such that the evaluation map \eqref{evaluation_m} is dominant is called \emph{minimal rational $m$-connecting degree} and is denoted by $d_{(X,H)}(m)$.
\end{proposition}
\begin{proof}
The variety $X$ is rationally connected and we want to show the existence of a morphism $\P^1 \to X$ whose image passes through a general set of $m$ points $p_1,\dots , p_m \in X$. 
This is \cite[Theorem~IV.3.9]{K_1996}.
\end{proof}

\begin{remark}\label{Rem:Bound_RC}
Let $(X,H)$ be a $d$-rationally connected polarised variety. Assume that $-K_X = i_X H$. Then the expected dimension of $\MmXd$ is $$d\cdot i_X + \dim(X) + m-3.$$
Moreover, the assumption of rational connectedness guarantees that this value is at least $m \cdot \dim(X)$, since the evaluation morphism \eqref{evaluation_m} is dominant. 
If the expected dimension is actually attained (this is true, for instance, for convex varieties), the following bound on $d$ is obtained:
$$d \geq \frac{(m-1)\dim(X) + 3 - m}{i_X}$$
and the same holds for the minimal rational $m$-connecting degree $d_{(X,H)}(m)$.
\end{remark}

The following result provides an inverse statement of Proposition~\ref{Prop:dm}: this will be useful to prove rational connectedness in examples.

\begin{lemma}\label{Lem:Reduction_degree}
Let $(X,H)$ be a polarised variety and fix an integer $d$. If there is an integer $m_d\ge 1$ for which $d\ge d_{(X,H)}(m_d)$ (i.e.~\eqref{evaluation_m} is dominant for $m=m_d$).
Then $d\ge d_{(X,H)}(m)$ also for all $1\le m < m_d$.
\end{lemma}

\begin{proof}
Fix $1\le m <m_d$ and look at the following commutative diagram:
\begin{equation}\label{CD_dominant}
\xymatrix{
\M_{0,m_d}(X,d) \ar[r]^-{\ev_{m_d}}
\ar[d]^-\varphi & X^{m_d} \ar[d]^-\pi \\
\M_{0,m}(X,d) \ar[r]^-{\ev_m} & X^{m}.\\}
\end{equation}
where the vertical arrow $\varphi$ is the morphism which forgets the last $m_d - m$ points, while $\pi$ is the projection on the first $m$ factors. 
Since both these arrows are surjective and the top horizontal arrow is dominant by hypothesis, so is the bottom one. 
\end{proof}

From now on, we assume $X$ to be rationally connected. As remarked in \cite[p.~1-2]{dJS_2006c}, the general fibre of the evaluation
$$\ev_2 \colon \MtwoXbeta \to X\times X$$
is an algebraic analogue of the space of 2-based paths in topology. Following this approach suggested by Barry Mazur, several reasonable definitions of \emph{rationally simply connected variety} were proposed 
by the authors of \cite{dJS_2006c} and \cite{dJHS_2011}, nonetheless one would like to identify \emph{one canonical} irreducible component of $\MtwoXd$ for sufficiently large $d$. Unfortunately, the following 
are some obstacles to a univocal definition: 

\begin{enumerate}
\item for a polarised variety $(X,H)$, the moduli spaces $\MmXd$ are disjoint unions of a priori many irreducible components: there can be several cycles $\beta$ which verify $(H\cdot \beta )= d$ and, for
any such $\beta$, one can have several components; 
\item even fixing one degree $d$ class $\beta$, the Kontsevich moduli spaces $\MmXbeta$ may have several irreducible components.
\end{enumerate}

\paragraph{\bf Solution to (1)} The polarisation $H\in \Pic(X)$ induces a 1-cycle class $H^{n-1} \in H_2(X,\Z)$. Let $\beta_{H}$ be the primitive element on the ray spanned by $H^{n-1}$ in $H_2(X,\Z)$
and we will only look at irreducible components of $\M_{0,m}(X,d\cdot \beta_H)$.

\medskip

\paragraph{\bf (Partial) solution to (2)} Fix a morphism $f_0\colon \P^1 \to X$ and let $\beta=[f_0(\P^1)] \in H_2(X,\Z)$ and $M_0$ an irreducible component of $\MzeroXbeta$ containing $[f_0]$. The forgetting morphism
$$\varphi=\varphi_{X,\beta}\colon \overline{M}_{0,1}(X,\beta)\to \MzeroXbeta$$
coincides with the universal family on $\MzeroXbeta$. We look at the restriction of $\varphi$ over the smooth locus $M_{0,\sm} \subset M_0$:
$$
\xymatrix{
M_{1, \sm} \ar@{^{(}->}[r] \ar[d]^{\varphi_{\sm}} & \overline{M}_{0,1}(X,\beta) \ar[d]^\varphi \\
M_{0, \sm} \ar@{^{(}->}[r]  & \MzeroXbeta. }
$$
Since $\varphi_{\sm}$ is a $\P^1$-bundle (locally trivial in the \'etale topology) by \cite[Theorem II.2.8]{K_1996}, the closure $M_1$ of $M_{1, \sm}$ in $\overline{M}_{0,1}(X,\beta)$ is irreducible. 
Iterating the argument, we can associate irreducible components $M_m \subset \MmXbeta$ for any integer $m\ge1$.

\medskip

We focus our attention on the following moduli subspace in $\M_{0,0}(X,\beta)$.
\begin{definition}
Let $\M_{0,0}^{\bir}=\MbirzeroXbeta$ denote the closure in $\MzeroXbeta$ of the moduli space of morphisms which are birational onto their image. 
The moduli subspace in $\MmXbeta$ dominating $\M_{0,0}^{\bir}$ will be denoted by $\M_{0,m}^{\bir}=\M_{0,m}^{\bir}(X,\beta)$.
\end{definition}

The space $\M_{0,0}^{\bir}$ is not irreducible in general: a nice simple example is given by $\M^{\bir}_{0,0}(X_1,-K_{X_1})$, where $X_1$ is a del Pezzo surface of degree 1 (cf.~\cite[Proposition~2.9]{T_2005}). 
Nonetheless, the previous discussion shows how to associate a unique irreducible component in $\MmXbeta$ for any irreducible component of $\M_{0,0}^{\bir}$. 
Our definition of rational simple connectedness will require some irreducibility of $\M_{0,0}^{\bir}$.

\medskip

We will analyse a special class of rationally connected varieties, which is central from the perspective of the birational classification of algebraic varieties. 
\begin{definition}\label{Def:Fano}
A smooth projective $k$-variety $X$ of dimension $n$ is a \emph{Fano variety} if its anticanonical bundle $\omega^{\vee}_X = \O_X(-K_X)$ is ample. 
\end{definition}
Let $i_X$ denote the \emph{Fano index} of $X$, i.e. the largest integer that divides $-K_X$ in the Picard group $\Pic(X)$.
Write $-K_X = i_X H_X$, for some ample divisor $H_X$; this gives a \emph{canonical polarisation} of $X$. Moreover the \emph{degree} of $X$ is defined as the top self intersection $(H_X)^n$.
It is well known that Fano varieties are rationally connected (cf. \cite{C_1992}, \cite{KMM_1992}).

Form now on, we will assume $(X,H)$ to be a canonically polarised Fano variety. In view of the previous discussion, we propose the following definition of rationally simply connected Fano variety.

\begin{definition}\label{Def:RSC}
Let $(X,H)$ be a canonically polarised Fano variety. One says that $X$ is {\bf canonically rationally simply connected} if there exists an integer $d_2>0$ such that, for all $d\ge d_2$, the following holds:
\begin{itemize}
\item the moduli space $\M_{0,0}^{\bir} \subset \M_{0,0}(X,d\cdot\beta_H)$ is irreducible;
\item the evaluation map 
\begin{equation}\label{evaluation_bir_2}
\ev_2\colon \M_{0,2}^{\bir} \to X\times X
\end{equation}
is dominant and its general fibre is rationally connected.
\end{itemize}
\end{definition}

\begin{remark}\label{Rem:MFS}
Let us invest few words to justify our definition. The notion of rationally simply connected varieties was introduced in \cite{dJS_2006c} and successfully applied in \cite{dJHS_2011} in the case of rationally connected varieties
of \emph{Picard rank one} to prove the following statement, known as \emph{Serre's Conjecture II}:

\medskip

\emph{Let $G$ be a semi-simple, simply connected algebraic group over $\C$, then any $G$-torsor over a smooth surface is locally trivial in the Zariski topology.} 

\medskip

In this sense, rational simple connectedness of the general fibre of a fibration over a surface should, at least in principle, guarantee the existence of rational sections (one also has to require the vanishing of some 
Brauer classes coming from the base, cf. \cite[Theorem~1.1]{dJHS_2011} and the discussion following the statement).

\medskip

Form the perspective of the birational classification of algebraic varieties and the minimal model program, the analysis of uniruled varieties (after \cite[Corollary~1.3.2]{BCHM_2010}) and conjecturally 
all varieties with negative Kodaira dimension (cf.~the \emph{Abundance Conjecture} \cite[Conjecture~3.12]{KM_1998}) can be reduced to the study of \emph{Mori fibre spaces}. 
These are special Fano fibrations with relative Picard rank 1, whose general fibre $X$ has stark geometric restrictions: the monodromy action on $\Pic(X)$ only preserves the ray generated by 
the canonical polarisation (cf.~\cite{M_1982}, \cite{CFST_2016} and \cite{CFST_2018} for more details). Dually in the space of curves, the ray of $\beta_H$ is the only one which is preserved by the monodromy action. 
From this perspective, $\M_{0,0}(X,d\cdot \beta_H)$ is the \emph{canonical choice}.
\end{remark}

Following the approach of \cite{dJS_2006c}, we provide a strong version of canonically rationally simply connected Fano variety.

\begin{definition}\label{Def:SRSC}
Let $(X,H)$ be a canonically polarised Fano variety. One says that $X$ is {\bf canonically rationally simply $m$-connected} if there exists an integer $d_m>0$ such that, for all $d\ge d_m$, the following holds:
\begin{itemize}
\item the moduli space $\M_{0,0}^{\bir} \subset \M_{0,0}(X,d\cdot\beta_H)$ is irreducible;
\item the evaluation map 
\begin{equation}\label{evaluation_bir_m}
\ev_m\colon \M_{0,m}^{\bir} \to X^m
\end{equation}
is dominant and its general fibre is rationally connected.
\end{itemize}
If $X$ is canonically rationally simply $m$-connected for all $m\ge 2$, then one says that $X$ is {\bf canonically strongly rationally simply connected}.
\end{definition}

\begin{remark}
Definitions \ref{Def:RSC} and \ref{Def:SRSC} can be simplified if $\Pic(X)=\Z$, simply taking as $H$ the primitive class. Nonetheless we will also study del Pezzo surfaces in Section~\ref{Sec:delPezzo},
whose moduli spaces of rational curves have been intensively studied (cf. \cite{T_2005}, \cite{T_2009}).
\end{remark}

Applying Lemma~\ref{Lem:Reduction_degree}, we deduce a criterion for strong rational simple connectedness.

\begin{proposition}\label{Prop:Reduction_degree}
Let $(X,H)$ be a rationally connected polarised variety and fix an integer $d$ and let $m_d\ge 1$ be an integer. If there is an irreducible component $M_{m_d} \subset \M_{0,m_d}(X,d)$
for which
\begin{equation}\label{evaluation_m_d}
\ev_{m_d}\colon M_{m_d} \to X^{m_d}
\end{equation}
is dominant and the general fibre is rationally connected, then the same holds for all $1\le m < m_d$.
\end{proposition}

\begin{proof}
Fix $1\le m <m_d$ and consider the diagram \eqref{CD_dominant} restricted to $M_{m_d}$:
$$\xymatrix{
M_{m_d} \ar[r]^-{\ev_{m_d}}
\ar[d]^-\varphi & X^{m_d} \ar[d]^-\pi \\
M_{m} \ar[r]^-{\ev_m} & X^{m}.\\}
$$
Thanks to Lemma~\ref{Lem:Reduction_degree}, we only have to prove the statement on rational connectedness.
Take a general point ${\bf x} \in X^m$: the fibre of $\ev_m^{-1}({\bf x})$ is dominated by $\ev_{m_d}^{-1}(\pi^{-1}({\bf x}))$, so it is enough to show that this last one is rationally connected. 
By hypothesis, $\pi^{-1}({\bf x})\simeq X^{m_d-m}$ and the general fibre of
$$\ev_{m_d}\colon \ev_{m_d}^{-1}(\pi^{-1}({\bf x})) \to \pi^{-1}({\bf x})$$
is rationally connected by hypothesis. So \cite[Corollary 1.3]{GHS_2003} implies that $\ev_{m_d}^{-1}(\pi^{-1}({\bf x}))$ is rationally connected too.
\end{proof}

\begin{corollary}\label{Cor:SuffRSC}
Let $(X,H)$ be a canonically polarised Fano variety and fix an integer $m\ge 2$. Let $d_0$ be an integer such that, for all $d\ge d_0$, there exists an integer $m_d\ge m$ which verifies the following conditions:
\begin{itemize}
\item $\M_{0,0}^{\bir} \subset \M_{0,0}(X,d\cdot\beta_H)$ is irreducible;
\item the evaluation map
\begin{equation}\label{evaluation_bir_m_d}
\ev_{m_d}\colon \M_{m_d}^{\bir} \to X^{m_d}
\end{equation}
is dominant and the general fibre is rationally connected.
\end{itemize}
Then $X$ is canonically rationally simply $m$-connected.
\\If moreover $m=2$ and $m_d$ is a strictly increasing function of $d$, then $X$ is canonically strongly rationally simply connected.
\end{corollary}

\begin{proof}
For the first part of the statement, we need to check that, for $d\ge d_0$, the evaluation map 
$$\ev_{m}\colon \M_{m}^{\bir} \to X^{m}$$
is dominant with rationally connected general fibre. This is a consequence of Proposition~\ref{Prop:Reduction_degree}, since $m \le m_d$. The second part of the statement follows.
\end{proof}

Unfortunately the previous corollary has very strong hypothesis and in order to study rational simple connectedness on $V_5$ we need a more careful study of the general fibre of the evaluation morphism.
\\For this purpose, we need to fix some further notation.

\begin{notation}\label{Not:GenFibre}
Let $(X,H)$ be a canonically polarised Fano variety and let $m\ge 0$ be an integer. Then we will often consider the fibre of the morphism
\begin{equation}
\Psi_m:= \phi_m \times \ev_{m} \colon \M_{0,m}^{\bir}(X,d) \to \M_{0,m}\times X^{m},
\end{equation}
where $\phi_m\colon\M_m^{\bir}(X,d) \to \M_{0,m}$ is the natural morphism to the Deligne-Mumford moduli space of $m$-marked rational curves.
\\Let ${\bf t} = (t_1,\ldots,t_m)\in (\P^1)^m$ be a marking on $\P^1$. The corresponding class in $\M_{0,m}$ will be also denoted by ${\bf t}$, to simplify the notation.
Moreover, fix $m$ points ${\bf x} = (x_1, \ldots , x_m)$ of $X$. Assume $m\ge 3$. Then the fibre 
$$\Psi_m^{-1}({\bf t},{\bf x})\cong_{\bir} \Mor^{{\bf t} \to {\bf x}}_d(\P^1,X)$$
is birational to the variety of degree $d$ morphisms $f\colon \P^1 \to X$ such that  $f({\bf t}) = {\bf x}$ which are birational onto the image.
If $m\le 2$, the previous relation holds modulo $\Aut(\P^1, {\bf t})$.
\end{notation}


\section{First examples: $\P^n$ and $\Q_n\subset \P^{n+1}$}\label{Sec:First_Examples}

In this section we study rational simple connectedness for some examples of Fano varieties with $\Pic=\Z$. All degree are with respect to the primitive generator of the Picard group. We will study here:

\begin{enumerate}
\item the projective space $\P^n$;
\item the quadric hypersurface $\Q_n \subset \P^{n+1}$;
\item rational homogeneous spaces $G/P$.
\end{enumerate}

For all these varieties, the moduli spaces $\M_{0,0}(X,d)$ (and hence the spaces $\M_{0,0}^{\bir}(X,d)$) are irreducible (cf. \cite{T_1998}, \cite{KP_2001}), so verifying canonical rational simple connectedness reduces to the study
of the evaluation morphisms.

\begin{remark}
For these varieties, our notion of canonical rational simple connectedness coincide with rational simple connectedness in \cite{dJS_2006c}. In particular:
\begin{itemize}
\item $\P^n$ and $\Q_n$ are strongly rationally simply connected (cf.~\cite[Theorems~1.2]{dJS_2006c});
\item rational homogeneous spaces are rationally simply connected (cf.~\cite[Corollary~3.8]{BCMP_2013}, \cite{dJHS_2011}).
\end{itemize}
\end{remark}

Before analysing examples, we should point out that the results on rational simple connectedness we present here for these first examples are well-known. Nonetheless, we push forward the analysis and obtain more precise information
on some moduli spaces of rational curves.

\subsection{The projective space}

For the projective space $\P^n$ it is easy to obtain the strongest possible statement for canonical rational simple connectedness.

\begin{proposition}\label{Prop:PnSRSC}
The projective space $\P^n$ is canonically strongly rationally simply connected.
\end{proposition}

Before proving the result, let us list some extra properties which hold for the moduli spaces of rational curves on the projective space:
\begin{itemize}
\item its VMRT (cf.~Remark~\ref{Rem:VMRT}) is isomorphic to $\P^{n-1}$;
\item the evaluation morphism $\ev_x$ from the universal $\P^1$-bundle over the VMRT
$$
\xymatrix {
  \P(\O_{\P^{n-1}}\oplus \O_{\P^{n-1}}(-1)) \ar[r]^-{\ev_x} \ar[d]_{\pi_x} & \P^n  \\
  \P^{n-1}  \ar@/_ 1pc/[u]_{\sigma_x}
}
$$
coincides with the blow up if $x \in \P^n$.
\end{itemize}

Let $f\colon \P^1 \to \P^n$ be a degree $d$ morphism, with $d\ge 1$, defined by $[P^0:\ldots : P^n]$, with $P^i\in \C[u,v]_d$ degree $d$ homogeneous polynomials for $i\in [0,n]$, fix $m=d+1$ \emph{distinct} points $t_0,\ldots,t_d \in \P^1$ 
and $m$ arbitrary points $x_0, \ldots , x_d \in \P^n$. We denote by $\Mor^{{\bf t} \to {\bf x}}_d(\P^1,\P^n)$ the variety of degree $d$ morphisms $f\colon \P^1 \to \P^n$ such that  $f(t_i) = x_i$ for all $i \in [0,d]$.

\begin{lemma}\label{Lem:MorPn}
Keep the notation as above. Then $\Mor^{{\bf t} \to {\bf x}}_d(\P^1,\P^n)$  is rational and isomorphic to
\begin{equation}\label{MorPn}
U_d = \{ [\lambda_0 : \ldots : \lambda_d] \ | \ \lambda_i \neq 0 \mbox{ for all } i \in [0,d] \} \subset \P^d.
\end{equation}
\end{lemma}

\begin{proof}
For $d = 1$, the variety $\Mor^{{\bf t} \to {\bf x}}_1(\P^1,\P^n)$ coincides with the automorphism group of $\P^1$ with two marked points (since there is a unique line through two points in $\P^n$), 
which is isomorphic to the multiplicative group $\G_m$ and is then rational. 
\\Assume now $d \geq 2$ and thus $m \geq 3$. In particular, the automorphism group of $\P^1$ preserving $m$ distinct points is trivial. Let $f\colon \P^1 \to \P^n$ be a degree $d$ morphism and 
choose coordinates on $\P^1$ such that $t_i = [z_i:1]$ and choose vectors $v_i\in \C^{n+1}$ such that $x_i = [v_i]$ for all $i \in [0,d]$.
In these coordinates we can write $f(t):=f([z:1]) = [P^0(z) : \ldots : P^n(z)]$. We impose now the conditions $f(t_i) = x_i$ for all $i \in [0,d]$, i.e. that there exist non-zero scalars $\lambda_i$, $i \in [0,d]$ 
such that 
$$P(z_i):=(P^0(z_i),\ldots,P^n(z_i)) = \lambda_i v_i.$$ 
Define
$$L_i(z) := \prod_{j \neq i}\frac{z - z_j}{z_i - z_j},$$
for all $i \in [0,d]$. We therefore have
\begin{equation}\label{PzMor}
P(z) = \sum_{i = 0}^d \lambda_i L_i(z) v_i.
\end{equation}
The variety the variety $\Mor^{{\bf t} \to {\bf x}}_d(\P^1,\P^n)$ is therefore described by \eqref{MorPn}.
\end{proof}

In Section~\ref{Sec:delPezzo} we will provide some results on del Pezzo surfaces, which can be obtained as blow-ups of $\P^2$ in points in general position. For this purpose, we need to study the fibres of some evaluation maps.
\\We fix $m=d+1$ and let $m'\ge 0$ be an integer. Let $\ev_{m+m'}$ be the evaluation map 
$$\ev_{m+m'}\colon \M_{0,m+m'}(\P^n,d) \to (\P^n)^{m+m'}.$$ 
As in Notation~\ref{Not:GenFibre}, we consider the map
$$\Psi_{m+m'}\colon \M_{0, m+m'}(\P^n,d) \to \M_{0,m+m'} \times (\P^n)^{m+m'}.$$
Let ${\bf x}=[{\bf v}]\in (\P^n)^{m+m'}$ be a fixed point, i.e. we write $x_i = [v_i]$ with $v_i\in \C^{n+1}$ for all $i \in [0,m+m']$.
The following lemma provides the numerical condition which guarantees that the general fibre of $\Psi_{m+m'}$ is non-empty and rational. It will be useful to study rationality for moduli spaces of rational curves on del Pezzo surfaces.

\begin{lemma}\label{Lemma:PnNonEmptyFibre}
Let $\U_{m+m'}\subset \M_{0,m+m'}$ the open subset of pairwise distinct points. Assume that the following two conditions hold:
\begin{enumerate}
\item $d\ge nm'$;
\item
\[\rk \begin{pmatrix}
\\
v_0\wedge v_{j}  & v_1\wedge v_{j} &  \cdots & v_d\wedge v_{j}\\
\\
\end{pmatrix}\]
\[= \rk \begin{pmatrix}
\\
v_0\wedge v_{j}   & \cdots & v_{i-1}\wedge v_{j} & v_{i+1}\wedge v_{j} & \cdots & v_d\wedge v_{j}\\
\\
\end{pmatrix}\]
for all $i\in[0,d]$ and $j \in [d+1,m']$.
\end{enumerate}
Then for any $({\bf t},{\bf x}) \in \U_{m+m'} \times (\P^n)^{m+m'}$ the fibre $\Psi_{m+m'}^{-1}({\bf t},{\bf x})$ is non-empty and rational.
\end{lemma}

\begin{proof}
The case $m'=0$ is simply Lemma~\ref{Lem:MorPn}. More precisely, a morphism $f\colon \P^1 \to \P^n$ sending the first $m$ pairwise distinct points $(t_0,\ldots,t_d)$ to $(x_0,\ldots,x_d)$ is described by \eqref{PzMor} and the variety $\Mor^{{\bf t} \to {\bf x}}_d(\P^1,\P^n)$ is rational of dimension $d$. Let $(z_{d+1},\ldots,z_{d+m'})$ be the coordinates of the last $m'$ points in $\P^1$. If we impose the extra condition (i.e.~$f(t_j)=x_i$  for all $j \in [d+1,d+m']$), we obtain the following condition on the coefficients $\pmb{\lambda}=(\lambda_0, \ldots, \lambda_d)$:
\begin{equation}\label{CondLambda}
\begin{bmatrix}
P^0(z_j) \\
P^1(z_j)  \\
\vdots\\
P^n(z_j)
\end{bmatrix}
=
\begin{bmatrix}
v_j^0 \\
v_j^1  \\
\vdots\\
v_j^n
\end{bmatrix}
\end{equation}
for all $j \in [d+1,d+m']$. These impose at most $nm'$ linear conditions on the $\lambda$'s.
We only need to check that these linear conditions are not of the form $\lambda_i=0$ for some $i$, since these would give an empty fibre $\Psi_{m+m'}^{-1}({\bf t},{\bf x})$. Here we need condition (2) in the hypothesis. For any $j \in [d+1,d+m']$, one can develop \eqref{CondLambda} and obtain:
\[ A\cdot \pmb{\lambda} :=
\begin{pmatrix}
\\
L_0(z_j)v_0 \wedge v_j & L_1(z_j)v_1 \wedge v_j & \cdots & L_d(z_j)v_d \wedge v_j\\
\\
\end{pmatrix}
\begin{pmatrix}
\lambda_0 \\
\lambda_1  \\
\vdots\\
\lambda_d
\end{pmatrix}
=
\begin{pmatrix}
0 \\
\vdots\\
0
\end{pmatrix}.\]
Condition (2) guarantees that $\Ker(A)$ is not contained in any of the hyperplanes $\lambda_i=0$, since $L_0(z_j)\neq 0$ for all $j \in [d+1,d+m']$.
\end{proof}

\begin{remark}\label{Rmk:PnNonEmptyFibreGen}
The condition (2) of Lemma~\ref{Lemma:PnNonEmptyFibre} holds if we assume that the points $(x_0,\ldots, x_{m+m'})$ are general: indeed, since $d\ge n$, we can choose the first $n+1$ vectors to be the standard basis, i.e.~$v_i=e_i$
for $i \in [0,n]$. Since for any $v\in \C^{n+1}$, the matrix $(e_0 \ e_1 \ \cdots \ e_n)\wedge v$ has rank at most $n$, the same holds for the matrix $(e_0 \ e_1 \ \cdots \ e_n \ v_{n+1} \ \cdots \ v_d)\wedge v$. Choosing $v_{n+1}, \ldots, v_d$ and $v$ general, condition (2) is verified.
\\Nonetheless, for our applications on del Pezzo surfaces, we cannot assume generality.
\end{remark}

\begin{corollary}\label{Cor:BoundPn}
Let $X:=\P^n$ be the $n$-dimensional projective space, with $n\ge2$. Then
$$m - 1 - \left\lfloor \frac{2(m-2)}{n+1} \right\rfloor \leq d_X(m) \leq m - 1 -\left\lfloor \frac{m-1}{n+1} \right\rfloor.$$
\end{corollary}

\begin{proof}
This is a direct consequence of Remark~\ref{Rem:Bound_RC} and Remark~\ref{Rmk:PnNonEmptyFibreGen}.
\end{proof}

\begin{proof}[Proof of Proposition~\ref{Prop:PnSRSC}]
We need to combine Lemma~\ref{Lem:MorPn} with Corollary~\ref{Cor:SuffRSC}. 
\\In the case $X=\P^n$, we have $\beta_H=l$, where $l$ is a line. Moreover, 
we know that all moduli spaces $\M_{0,0}(\P^n,d\cdot l)$ are irreducible (cf. \cite[Section~4]{FP_1997}). Lemma~\ref{Lem:MorPn} guarantees that $m_d=d+1$ verifies the hypothesis of Corollary~\ref{Cor:SuffRSC}.
More precisely, we consider the following diagram:
\begin{equation}\label{RCfibrePn}
\xymatrix{
\overline{M}_{0,d+1}(X,d) \ar[d]^-{g} \ar[r]^-{\ev_{d+1}} & X^{d+1} \\
\overline{M}_{0,d+1}}.
\end{equation}
Lemma~\ref{Lem:MorPn} implies that the general fibre $M_{\bf x}:=\ev_{d+1}^{-1}({\bf x})$ is dominant on $\overline{M}_{0,d+1}$  and that $g_{|M_{\bf x}}$ is birational to a $\P^d$-fibration over $\overline{M}_{0,d+1}$.
Since the base of $g_{|M_{\bf x}}$ is rational (cf. \cite{K_1993}), we deduce the unirationality of $M_{\bf x}:=\ev_{d+1}^{-1}({\bf x})$.

Since $m_d$ is strictly increasing, we conclude by Corollary~\ref{Cor:SuffRSC} that $\P^n$ is canonically strongly rationally simply connected.
\end{proof}

\subsection{The quadric hypersurface}

Also for this example, we obtain rational simple connectedness in the strongest possible sense. Nonetheless, the proof requires more care, since we need to distinguish two cases, depending on some parity condition.

\begin{proposition}\label{Prop:QnSRSC}
The quadric hypersurface $\Q_n \subset \P^{n+1}$, with $n\ge 2$, is canonically strongly rationally simply connected.
\end{proposition}

\begin{remark}
The case of $\Q_2=\P^1\times\P^1$ is slightly different with respect to the others, since $\rho(\Q_2)=2$. Nonetheless, we will always fix the diagonal polarisation $\beta =H_1 + H_2$, where the $H_i$'s are the two rulings. With this choice, all the arguments of this section apply also for $n=2$.
\end{remark}

As for $\P^n$ let us recall that, for any $x \in \Q_n \subset \P^{n+1}$, the VMRT is isomorphic to $\Q_{n-2}$, which is then rational, if $n>2$, or two reduced points if $n=2$.

As for the previous case, let $f\colon \P^1 \to \Q_n\subset \P^{n+1}$ be a degree $d$ morphism, with $d\ge 1$, defined by $[P^0:\ldots : P^{n+1}]$, with $P^i\in \C[u,v]_d$ for $i\in [0,n+1]$, fix $m=d+1$ \emph{distinct} points $t_0,\ldots,t_d \in \P^1$ 
and $m$ arbitrary points $x_0, \ldots , x_d \in \Q_n$. Keeping the analogous notation as before, we denote by $\Mor^{{\bf t} \to {\bf x}}_d(\P^1,\Q_n)$ the variety of degree $d$ morphisms $f\colon \P^1 \to \Q_n$ such that  $f(t_i) = x_i$ for all $i \in [0,d]$.

For this case, we need some more notation. Let $V$ be a $\C$-vector space of dimension $n+2$ and $q$ be a non-degenerate quadratic form on $V$ which defines the quadric $\Q_n\subset \P(V)$, i.e. 
$$\Q_n = \{ [v] \in \P(V) \ | \ q(v) = 0\}.$$
Moreover, let $B$ be the non-degenerate bilinear form associated to $q$.

As in the proof of Lemma~\ref{Lem:MorPn}, choose coordinates on $\P^1$ such that $t_i = [z_i:1]$ and choose vectors $v_i\in V$ such that $x_i = [v_i]$ for all $i \in [0,d]$. 

\begin{definition}\label{Def:MatA}
Keep the notation as above. The {\it rescaled skew-symmetric matrix} $A:=A_{\Q_n, {\bf z} \to {\bf v}}$ associated to $(z_0, \ldots, z_d)$ and $(v_0,\ldots, v_d)$ is defined as
$$A_{\Q_n, {\bf z} \to {\bf v}} := \left( \frac{B(v_i,v_j)}{z_j - z_i} \right)_{i,j \in [0,d]}.$$
\end{definition}

The previous definition depends on the choice of coordinates on $\P^1$, but this will not affect our arguments. Let us recall some properties of the Pfaffian of skew-symmetric matrices.

\begin{remark}\label{Rem:pf(A)}
Let $A=\{a_{ij}\}$ be an $(m\times m)$ skew-symmetric matrix. The {\it Pfaffian} $\pf(A)$ of $A$ is defined as
\begin{equation}\label{PfA}
\pf(A):=\begin{cases}
0 \ \ \ \ \text{ if $m$ is odd;}\\ 
 \sum_{\sigma \in \FF_{m}} \sgn(\sigma)\prod_{i=1}^k a_{\sigma(2i-1),\sigma(2i)}     \ \ \ \ \text{ if $m=2k$ is even.}
\end{cases}
\end{equation}
where $\FF_{m}$ is the set of permutations in $\SS_{m}$ satisfying the following:
\begin{itemize}
\item $\sigma(1)<\sigma(3)<\ldots < \sigma(2k-1)$; and
\item $\sigma(2i - 1) < \sigma(2i)$ for $1\le i \le k$.
\end{itemize}
One can check that the set $\FF_{2k}$ is in 1:1 correspondence with partitions of the set  $\{1, 2, \ldots , 2k\}$ into $k$ disjoint subsets with 2 elements:
$$\sigma \in \FF_{2k} \xleftrightarrow{\ 1:1\ } \{(i_1, j_1), \ldots , (i_k, j_k)\}$$
where $i_1<i_2<\ldots < i_k$ and $i_l < j_l$ for all $1 \le l \le k$. In particular,
$$|\FF_{2k}|=\frac{(2k)!}{2^kk!}=(2k-1)!!$$

So, $\pf(A)$ is a degree $k$ polynomial in the entries of $A$ and is linear in its lines and columns. Moreover the following formal properties of $\pf(A)$ hold:
\begin{itemize}
\item $\pf(A)^2=\det(A)$;
\item $\pf(BAB^{t})=\det(B)\pf(A)$ for any $(m\times m)$-matrix $B$.
\end{itemize}
Moreover a change of coordinates on $\P^1$ will change the Pfaffian 
by a non-zero scalar and the vanishing of $\pf(A)$ will not depend on the choice of coordinates. For further details on pfaffians, see \cite[Section~5.7]{N_1984}.
\end{remark}
We can prove an analogue of Proposition~\ref{Prop:PnSRSC} for quadric hypersurfaces. 
\\Let us recall some known facts.
In our notation, given a matrix $A = (a_{i,j})$, the $i$-th row (resp.~ the $j$-th column) is denoted by $A_i$ (resp.~ by $A^j$).

\begin{lemma} \label{Lem:PfComp}
Let $A = (a_{i,j})$ be an ($m\times m$)-skew-symmetric matrix, with indices $i,j \in \{0,1,\ldots, d\}$. Let $A(i)$ (resp.~ $A(i,j)$) denote the skew-symmetric matrix obtained from $A$ by removing the $i$-th row and column (resp.~ the $i$-th and the $j$-th rows and columns).
\begin{enumerate}
\item if $m$ is odd and $A$ has rank $m-1$, then $\Ker(A)$ is spanned by the vector $((-1)^1\pf(A(0)),\cdots,(-1)^{d+1}\pf(A(d)))^t$ .
\item If $m=2k$ is even and $A$ has rank $m-2$, then $\Ker(A)$ is spanned by the $m$ vectors
$$N_i = ((-1)^{i+j}\pf(A(i,j)))_j;$$
\end{enumerate}
\end{lemma}

\begin{proof}
We prove (1): let us define $v_A := ((-1)^1\pf(A(0)),\cdots,(-1)^{d+1}\pf(A(d)))$. Let $A[i]$ denote the following matrix:
$$A[i]:=
\begin{pmatrix}
A & A^i\\
A_i & 0
\end{pmatrix},
$$
which is skew-symmetric of dimension $(m+1)\times (m+1)$, and by construction $\pf(A[i])=0.$
Using the formal properties of the Pfaffian (see \cite[Lemma~2.3]{IW_1999}), we obtain:
$$0 = \pf(A[i]) = \sum_{j = 0}^d (-1)^{i+j} a_{i,j} \pf(A(j)) = (-1)^{i-1} A_i \cdot v_A^t.$$
In particular $v_A^t$ is in the kernel of $A$ and since $A$ has co-rank one, at least one of the entries of $v_A^t$ is non-zero and so $v_A$ generates the kernel.

To prove part (2), we apply part (1) to deduce that all vectors $N_i$ are in the kernel of $A$ (recall that $\pf(A(i,i)) = 0$). Furthermore, the vectors $N_i$ and $N_j$ are linearly independent as soon as $\pf(A(i,j)) \neq 0$. But the condition on the rank
implies that at least one of these Pfaffians is non-zero and that the $N_i$'s span the kernel of $A$.
\end{proof}

\begin{remark}
The same methods can be used to deduce similar results on the kernel of skew-symmetric matrices for higher co-ranks.
\end{remark}

Applying the formal properties of pfaffians to our setting, we deduce the following lemma.

\begin{lemma}\label{Lem:maxRk}
Keep the notation of Definition~\ref{Def:MatA}. Let $u, v \in V$ be vectors such that $q(u)=q(v)=0$ and $B(u,v) = 1$. Let ${\bf v}=(v_0,\ldots, v_d)$ be defined via
$$v_i=
\begin{cases}
u \text{ for } i \text{ even;}\\
v \text{ for } i \text{ odd}
\end{cases}
$$
with $i \in [0,d]$. Let $\U_m\subset \M_{0,m}$ the open subset of pairwise distinct points. Then for any ${\bf t} \in \U_m$ (with coordinates ${\bf z}$) the matrix $A=A_{\Q_n, {\bf z} \to {\bf v}}$ has maximal rank.
\end{lemma}
\begin{proof}
 The associated matrix $A$ has the form
$$A = \left( \frac{\delta_{i + j \equiv 1}}{z_j - z_i} \right),$$
where $\delta_{i + j \equiv 1}$ equals $1$ if $i + j$ is odd and $0$ otherwise. For $d$ even, we have $\pf(A) = 0$ while for $d = 2k - 1$ odd, we prove  the following formula:
$$\pf(A) =\frac{\prod_{i < j,\ i + j \even}(z_i - z_j)}{\prod_{i < j,\ i + j \odd}(z_i - z_j)}.$$
Indeed, keeping the notation as in Remark~\ref{Rem:pf(A)}, let $\FF_{2k}^{\odd}\subset \FF_{2k}$ be the subset of partitions $\sigma=\{(i_1, j_1), \ldots , (i_k, j_k)\}$ such that $i_l + j_l$ is odd for all $l \in[1, k]$. 
Using formula \eqref{PfA}, one has:
\begin{equation}\label{pf_lemma}
\pf(A) = \sum_{\sigma \in \FF_{2k}^{\odd}}\prod_{l = 1}^k \frac{1}{z_{j_l} - z_{i_l}}.
\end{equation}
Factorising this expression, we deduce:
$$\pf(A) = \frac{P(z_0,z_1,\ldots,z_d)}{\prod_{i < j,\ i + j \odd}(z_j - z_i)}.$$
The denominator has degree $k^2$ and the polynomial $P(z_0,z_1,\ldots,z_d)$ has degree at most $k^2-k$. If we evaluate $P$ in $z_i = z_j$ for $i<j$ and  $i + j $ even, the numerator vanishes, so there exists a constant $C$ for which $P(z_0,z_1,\ldots,z_d)=C\prod_{i < j,\ i + j \even}(z_j - z_i)$. One determines the value of $C=1$ looking at the residue at $(z_0-z_1)$.
\end{proof}

The following lemma is the quadratic version of Lemma~\ref{Lem:MorPn}.

\begin{lemma}\label{Lem:MorQn}
Keep the notation as above. Then the variety $\Mor^{{\bf t} \to {\bf x}}_d(\P^1,\Q_n)$, with for $d\ge 2$, is either empty or rational and isomorphic to the open subset of $\P^d$ defined by
\begin{equation}\label{MorQn}
U_{d,A} = \{ [\lambda_0 : \ldots : \lambda_d] \ | \ \lambda_i \neq 0 \mbox{ for all } i \in [0,d] \text{ and } (\lambda_0,\ldots,\lambda_d) \in \Ker(A) \}.
\end{equation}
\end{lemma}

\begin{proof}
Since $\Q_n$ is contained in the projective space $\P^{n+1}$ we can apply Lemma~\ref{Lem:MorPn} to write down the parametrisation of morphisms in $f\colon \P^1 \to \P^{n+1}$ as in \eqref{PzMor}.
\\We only need to impose the condition that $f$ factors through $\Q_n$. Let us define
\begin{equation}\label{QzMor}
Q(z) := q(P(z)) = B(P(z),P(z)).
\end{equation}
By assumption, we have that $F(z_i)=0$ for all $i$, so imposing that $f$ factors through $\Q_n$ is equivalent to requiring that $Q$ vanishes with multiplicity at least $2$ at the $z_i$'s.
Computing the differential $dQ(z) = 2 B(dP(z),P(z))$ we have
$$dQ(z) = 2 \sum_{i,j = 0}^d dL_i(z) L_j(z) \lambda_i\lambda_j B(v_i,v_j) = 2 \sum_{i,j = 0, \ i \neq j}^d dL_i(z) L_j(z) \lambda_i\lambda_jB(v_i,v_j),$$
where the last equality holds because $B(v_i,v_i) = 0$. 
An easy computation shows the following formula, for any $l\neq i$:
$$dL_i(z_l) = \frac{1}{z_l - z_i} \frac{\zeta_l}{\zeta_i},$$
where $\zeta_i = \prod_{k \neq i}(z_i - z_k)$. So evaluating the differential $dQ(z)$ in $z_l$'s we obtain:
$$dQ(z_l) =  2 \sum_{i = 0, \ i \neq l}^d dL_i(z_l) \lambda_l\lambda_i B(v_i,v_l) $$ $$ = 2 \lambda_l\zeta_l \sum_{i = 0, \ i \neq l}^d \frac{\lambda_i}{\zeta_i} \frac{B(v_i,v_l)}{z_l - z_i} 
=  2 \lambda_l \zeta_l A_l \cdot \left(\frac{\lambda_0}{\zeta_0},\cdots,\frac{\lambda_d}{\zeta_d}\right)^t,$$
where $A_l$ is the $l$-th row of $A$. So $dQ$ vanishes in the $z_l$'s if and only if the vector $(\lambda_0/\zeta_0,\cdots,\lambda_d/\zeta_d)$ is in $U_{d,A}$. 
\end{proof}

Our final aim in this section is to prove rational connectedness of the general fibre of some evaluation morphisms for $\Q_n$: for this, it is enough to look at the spaces $\Mor^{{\bf t} \to {\bf x}}_d(\P^1,\Q_n)$ for general ${\bf z}$ and ${\bf x}$ and for sufficiently large $d$. Nonetheless, in Section~\ref{Sec:V5_1} we will need some more punctual information about {\it special} fibres of evaluation maps for $\Q_n$. 

Assume that $m=d+1$ and consider the evaluation map 
$$\ev_m\colon \M_m^{\bir}(\Q_n,d) \to (\Q_n)^m,$$ 
as in Notation~\ref{Not:GenFibre}. The following lemma describes the general fibres of $\Psi_m$.

\begin{lemma}\label{Lemma:MorQn}
Keep the notation as above. 
\begin{enumerate}
\item If $d$ is odd and ${\bf t}$ and ${\bf x}$ are general, $\Mor^{{\bf t} \to {\bf x}}_d(\P^1,\Q_n)$ is empty;
\item if $d$ is even and ${\bf t}$ and ${\bf x}$ are general, $\Mor^{{\bf t} \to {\bf x}}_d(\P^1,\Q_n)$ is a point.
\end{enumerate}
\end{lemma}

\begin{proof}
Let us consider the matrix $A=A_{\Q_n, {\bf z}}$, where we chose coordinates $t_i = [z_i:1]$ on $\P^1$. Then the result follows from Lemma~\ref{Lem:MorQn} and the following claim (in italics).
{\it \begin{enumerate}
\item if $d$ is odd and ${\bf t}$ and ${\bf x}$ are general, the matrix $A$ has maximal rank;
\item If $d$ is even and ${\bf t}$ and ${\bf x}$ are general, the matrix $A$ has rank $d$ and $\Ker(A)$ is spanned by a vector with non-zero coordinates.
\end{enumerate}}
Part (1) is a consequence of Lemma~\ref{Lem:maxRk}. Let us show Part (2). Since $A$ is of odd dimension $m = d+1$, it is degenerate. Let $A(0),\cdots,A(d)$ denote the skew-symmetric matrices as in Lemma~\ref{Lem:PfComp}. Using Part (1) of the claim, for general ${\bf t}$ and ${\bf x}$, we have 
$\pf(A(i)) \neq 0$ for all $i$. In particular $A$ has rank $d=m-1$ and, by Lemma~\ref{Lem:PfComp},  its kernel is spanned by $((-1)^1\pf(A(0)),\cdots,(-1)^{d+1}\pf(A(d)))$ which is a vector with non-zero coordinates. This proves (2).
\end{proof}

\begin{corollary}\label{Cor:BoundQn}
Let $X:=\Q_n$ be a $n$-dimensional quadric. Then $d_X(2) = 2$ and for $m \geq 3$, 
$$m - 1 - \left\lfloor \frac{m-3}{n} \right\rfloor \leq d_X(m) \leq m - 1.$$
\end{corollary}

\begin{proof}
The first inequality is a consequence of Remark~\ref{Rem:Bound_RC}.
For $m = 2$, it is well known that there is no line but a conic through two general points on $\Q_n$. 
\\Assume $m \geq 3$ and fix $m$ points on $\Q_n$ and let $d = m-1$. If $m$ is odd, Lemma~\ref{Lemma:MorQn} gives the result. For even $m \geq 4$, let us define $m' := m-1$ and $d' := d-1$. For a general ${\bf x'} := (x_1,\ldots, x_{m'})\in (\Q_n)^{m'}$, again Lemma~\ref{Lemma:MorQn} implies that there exists a degree $d'$ morphism $f\colon \P^1 \to \Q_n$ through ${\bf x'} $. Now let $x_m$ be a general point in $\Q_m$. Then the hyperplane $H_{x_m}$ dual to $x_m$ meets $f(\P^1)$ in at least a point $y$ and the line $l$ passing through $x_m$ and $y$ is contained in $\Q_m$ by construction. In particular the points $(x_1,\ldots, x_{m'},x_m)$ are on the image of a degree $d$ morphism $g\colon \P^1 \cup \P^1 \to \Q_m$ with image $f(\p^1) \cup l$. Since the moduli space of stable maps to $\Q_n$ is irreducible (cf.~\cite[Corollary~1]{KP_2001}) and it has a dense open subset of maps from an irreducible curve, we may find a degree $d$ irreducible rational curve passing through general $x_1,\ldots, x_m$ in $\Q_n$.
\end{proof}

The following proposition can be seen as a refinement of Lemma~\ref{Lemma:MorQn} and will be crucial in Section~\ref{Sec:V5_1}.

\begin{proposition}\label{Prop:GenRkQn}
Let $\gamma$ be an integral curve in $\Q_n$ which is non-degenerate in $\P^{n+1}$, i.e.~$\langle \gamma\rangle =\P^{n+1}$. For the product $Z:=\gamma^m \subset (\Q_n)^m$, consider the fibre $M:=\ev^{-1}_m(Z)$. Then, for any irreducible component $M'$ of $M$, the image $\Psi_m(M')$ contains a point $({\bf t}, {\bf x})\in \M_{0,m} \times Z$ such that the corresponding matrix $A=A_{\Q_n, {\bf z} \to {\bf v}}$, has rank at least $2 \left \lceil\frac{d-1}{2} \right \rceil$.
\end{proposition}

\begin{proof}
First, notice that by Lemma~\ref{Lemma:MorQn}, the value $2 \left \lceil\frac{d-1}{2} \right \rceil$ is the maximal possible rank for the matrix $A_{\Q_n, {\bf z} \to {\bf v}}$, with $({\bf t}, {\bf x})\in \Psi_m(M)$.
\\Since the variety $\M_{0,m}(\Q_n,d)=\M_{0,m}^{\bir}(\Q_n,d)$ is irreducible of dimension $(n+1)(d+1)-3$ (cf. \cite[Corollary~1]{KP_2001}) and $M$ is locally defined by $(n-1)(d+1)$ equations, coming from the local equations of $\gamma$ in $\Q_n$, the irreducible component $M'$ has dimension at least $2d -1$.

Assume that for a general element $p \in M'$, the matrix $A=A_{\Q_n, {\bf z} \to {\bf v}}$ associated to $({\bf t}, {\bf x}) = \Psi_m(p)$ has rank $2r$. 
\\By Lemma~\ref{Lem:MorQn} the general fibre of $\Psi_m\colon M' \to \Psi_m(M')$ has dimension $d - 2r$ and $\Psi_m(M')$ has dimension at least $d+2r-1$. Consider the projection on $Z$
$$\pr_2 \colon \Psi_m(M') \to Z,$$
and let us look at the dimension of its (nonempty) fibres.

Let $K := \{\kappa_1, \ldots , \kappa_{2r}\}\subset [0,d]$  with $\kappa_1 < \cdots < \kappa_{2r}$ be indices such that $\pf(A_K^K) \neq 0$ where $A_K^K$ is the matrix obtained form $A$ by the removing the lines and columns of indices outside of $K$. \\Fix $\kappa_0 \in [0,d] \setminus K$ and reorder $[0,d]$ so that $[0,d] = \{\kappa_0,\kappa_1, \ldots, \kappa_{2r}, \kappa_{2r+1}, \dots, \kappa_{d}\}$. 
Let us define $K_0:=[0,2r]$ and $K_{s}:=K_0 \cup \{\kappa_s\}$ for any $s \in [2r + 1,d]$.
By definition, the element $({\bf t}, {\bf x})$ satisfies the equations
\begin{equation}\label{E_s}\tag{$E_s$}
\pf(A_{K_s}^{K_s}) = 0 \textrm{ for all $s \in [2r + 1,d]$}.
\end{equation}
If $B(v_{\kappa_0},v_{\kappa_s}) \neq 0$, expanding $\pf(A_{K_s}^{K_s})$ with respect to the line $\kappa_s$, we get (cf. the notation of Lemma~\ref{Lem:PfComp} and \cite[Lemma~2.3]{IW_1999}):
$$0 = \sum_{j=0}^{2r}(-1)^{\kappa_s+j}a_{\kappa_s j} \pf(A_{K_0}^{K_0}(j))$$
$$=(-1)^{\kappa_s}a_{\kappa_s \kappa_0}\pf(A_{K}^{K}) + (-1)^{\kappa_s}\sum_{j=1}^{2r}(-1)^ja_{\kappa_s j} \pf(A_{K_0}^{K_0}(j))$$
$$=(-1)^{\kappa_s}\frac{B(v_{\kappa_s},v_{\kappa_0})}{z_{\kappa_0} - z_{\kappa_s}}\pf(A_{K}^{K}) + (-1)^{\kappa_s}\sum_{j=1}^{2r}(-1)^j\frac{B(v_{\kappa_s},v_{j})}{z_{j} - z_{\kappa_s}} \pf(A_{K_0}^{K_0}(j)).$$

So the previous equation is nontrivial and  the term $a_{\kappa_s \kappa_0}\pf(A_{K}^{K})$
is the only term contributing for a pole along $(z_{\kappa_0} - z_{\kappa_s})$. 
So all the nontrivial equations $\eqref{E_s}$ with $s \in [2r + 1,d]$, are independent, since the equation $\eqref{E_s}$ only involves the variables $(z_{\kappa_i})_{i \in K_0}$ and $z_{\kappa_s}$. 

Let $l$ is the number indices $s \in [2r+1,d]$ such that $B(v_{\kappa_0},v_{\kappa_s}) = 0$. Since $\gamma$ is non-degenerate, we get this way $d - 2r - l$ independent equations and hence 
$$\dim \pr_2^{-1}(\pr_2({\bf t}, {\bf x})) \leq (d-2) - (d - 2r - l) = 2r + l -2.$$  

Up to reordering the indices, we can choose $\kappa_0$ so that $l$ is minimal. In particular for each $\tilde \kappa \in [0,d] \setminus K$, there are at least $l$ vanishings $B(v_{\tilde\kappa},v_{\kappa_s}) = 0$ for $s \in \{0\}\cup [2r+1,d]$. 
This in particular implies that we have at least $\frac{l(d+1-2r)}{2}$ equations of the form
$$B(v_i,v_j) = 0 \textrm{ for } i,j \in [0,d] \setminus K \textrm{ and } i \neq j.$$
This in turn implies that the dimension of the image $\pr_2(\Psi_m(M'))$ verifies 
$$\dim(\pr_2(\Psi_m(M')))\le \dim Z - \frac{l(d+1-2r)}{2} = d + 1 - \frac{l(d+1-2r)}{2},$$
and finally that 
$$\dim (\Psi_m(M')) \le d + 1 - \frac{l(d+1-2r)}{2} + 2r + l - 2 = d + 2r -1 - l \frac{d - 1 - 2r}{2}.$$
Since $\Psi_m(M')$ has dimension at least $d + 2r - 1$, we get $l (d - 1 - 2r) \leq 0$, i.e. either 
\begin{itemize}
\item $\rk (A) \geq 2\left \lceil \frac{d-1}{2} \right \rceil$; or
\item $l = 0$.
\end{itemize}
We need to treat the second case, so let assume that $l=0$. The above estimates imply that $\pr_2$ is surjective. We now produce an element ${\bf x} \in Z$ such that for any ${\bf t} \in \M_{0,m}$, the corresponding matrix $A$ has maximal rank. Indeed, since $\gamma$ is non-degenerate, we can find two vectors $u, v \in V$ such that $[u], [v] \in \gamma$ and $B(u,v) = 1$. We conclude applying Lemma~\ref{Lem:maxRk}.

We have proved that any component $M'$ of $M$ is such that $\Psi(M')$ contains an elements whose associated matrix $A$ has rank at least $2\left\lceil \frac{d-1}{2} \right\rceil$.
\end{proof}

\begin{proof}[Proof of Proposition~\ref{Prop:QnSRSC}]
Using Corollary~\ref{Cor:BoundQn}, we see that, for $m=d+1$ the evaluation maps are dominant. To show that the general fibre $M_{\bf x}=\ev_m^{-1}({\bf x})$ is rationally connected, we have two cases. If $d$ is even, Lemma~\ref{Lemma:MorQn} implies that the map $\psi_m\colon M_{\bf x} \to \psi_m(M_{\bf x})=\M_{0,m}$ is birational. This implies that $M_{\bf x}$ is rational (cf. \cite{K_1993}). If $d$ is odd, we use Lemma~\ref{Lem:maxRk} and Lemma~\ref{Lem:MorQn} and deduce that the general fibre of $\psi_m\colon M_{\bf x} \to \psi_m(M_{\bf x})=\M_{0,m}$ is birational to $\P^1$. Moreover, Lemma~\ref{Lem:PfComp} implies that this morphism has a rational section given by $N_0 = ((-1)^{j}\pf(A(0,j)))_j$. So $M_{\bf x}$ is rational also in this case.
To obtain strong rational simple connectedness of $\Q_n$, we can apply Corollary~\ref{Cor:SuffRSC} (irreducibility of all moduli spaces $\M_{0,0}(\Q_n)$ is in \cite[Corollary~1]{KP_2001}).
\end{proof}

\subsection{Further examples, results and questions}

We start with a question about rational curves on $\P^n$ and $\Q_n$. 

\begin{question}
Do the following equalities hold?
\begin{enumerate}
\item $d_{\P^n}(m) = m - 1 - \left\lfloor \frac{2(m-2)}{n+1} \right\rfloor;$
\item $d_{\Q_n}(m) = m - 1 - \left\lfloor \frac{m-3}{n} \right\rfloor.$
\end{enumerate}
\end{question}

There are some classes of varieties, which have been studied from the perspective of rational simple connectedness. We discuss here some known results.

\subsubsection{Rational homogeneous spaces} 
Let us consider the case $X = G/P$, where $G$ a reductive group and $P$ a maximal parabolic subgroup. After the work \cite{KP_2001}, in which the irreducibility and rationality of the moduli spaces $\M_{0,n}(X,\beta)$ has been established, the geometry of rational curves in rational homogeneous spaces has been studies in several recent works (cf. \cite{LM_2003}, \cite{dJHS_2011}, \cite{BCMP_2013}).
\\ The VMRT of those varieties is given by lines and the variety of lies of $X$ is actually rational: a very explicit description of the variety of lines of $X$ appeared in \cite[Theorem~4.3]{LM_2003}). 

The variety $X$ is also (canonically) rationally simply connected (cf. \cite[Corollary 3.8]{BCMP_2013}, see also \cite[Section~15]{dJHS_2011}).

\begin{question}\label{Quest:Hom}
Let $X = G/P$ be  rational homogeneous variety. We state two open problems.
\begin{enumerate}
\item Is $X$ canonically rationally simply 3-connected?
\item More generally, is $X$ canonically strongly rationally simply connected?
\end{enumerate}
\end{question}

Question~\ref{Quest:Hom}(1) has interesting consequences from the perspective of quantum cohomology (cf. \cite[Remark~3.4]{BCMP_2013}, \cite{BCMP_2016}). Moreover the arguments in \cite[Section~15]{dJHS_2011} seem to lead to stronger statements, as Jason Starr pointed out (cf. \cite[Remark~3.5]{BCMP_2013}).

\subsubsection{Complete intersections} 

Looking at the examples we discussed previously, it is natural to try to study rational simple connectedness for complete intersections in projective spaces. This has been studied in \cite{dJS_2006c} and \cite{D_2015}. 
Their notion of rational simple connectedness is slightly different compared with the one in Definitions~\ref{Def:RSC} and \ref{Def:SRSC}, nonetheless, at least for a {\it general} complete intersections the different notations
coincide (cf. \cite{HRS_2004} and \cite{BM_2013} for the irreducibility of $\M_{0,0}^{\bir}$ for complete intersections of low degree). We recall here one of their main results.
\begin{theorem}[de Jong-Starr, DeLand]
Let $X_{\bf d} \subset \P^n$ be a general complete intersection of degree ${\bf d}=(d_1,\ldots , d_c)$, $d_i\ge 2$, with $\dim(X_{\bf d})\ge 3$. Then $X_{\bf d}$ is (canonically) strongly rationally simply connected if and only if
$$\sum d_i^2 \le n.$$
\end{theorem}


\section{Del Pezzo surfaces}\label{Sec:delPezzo}

In this section we study our notion of simple rational connectedness for surfaces. This case has been almost completely ignored by other authors, since del Pezzo surfaces have high Picard rank.
This is the main reason why our notions in Definitions~\ref{Def:RSC} and \ref{Def:SRSC} require the choice of a canonical polarisation.

\begin{notation}
Let $X=X_\delta$ be a del Pezzo surface of degree $\delta$, with $X_\delta \not\simeq \P^1\times \P^1$. We will always fix a geometric marking $\Pic(X_\delta) \simeq \Z^{10-\delta}$, i.e.~ a realisation of $X_\delta$ as blow up $\pi\colon X_\delta \to \P^2$ with centre the points $p_1,\ldots, p_{9-\delta}$. Every class $\beta$ on $X_\delta$ can then be (uniquely) written as $\beta = a\pi^*H - \sum_{i=1}^{9-\delta}b_i E_i$. In this context, we will write $b:=\sum_{i=1}^{9-\delta}b_i$ and $\beta = (a; b_1, \ldots, b_{9-\delta})$.
\end{notation}

Let us state our main result.

\begin{theorem}\label{Thm:dPRSC}
Let $X_\delta$ be a del Pezzo surface of degree $\delta\ge 5$. Then $X_\delta$ is canonically strongly rationally simply connected.
\end{theorem}

The 2-dimensional projective space $\P^2$ and the quadric surface $\P^1\times\P^1$ with diagonal polarisation have been studied in the previous section, so we will focus on the other del Pezzo surfaces.

We will deduce some precise results on moduli spaces of rational curves on del Pezzo surfaces with quite general polarisation, so let us recall some facts about the structure of numerical cones $\Amp(X) \subset \Nef(X)$.

\begin{lemma}\label{Lem:ConesdP}
Let $X=X_\delta$ be a del Pezzo surface of degree $\delta$, $X_\delta \not\simeq \P^2,\P^1\times\P^1$. Then a class $\beta = (a; b_1, \ldots, b_{9-\delta})$ is nef iff the following inequalities hold. 
\begin{description}
\item[$\pmb{\delta\ge 7}$] $a\ge b\ge 0$;
\item[$\pmb{\delta=5,6}$]
\
\begin{itemize}
\item $b_i\ge 0$ $\forall i$,
\item $a \ge b_i +b_j$, $\forall i>j$.
\end{itemize}
\end{description}

Moreover, the class $\beta$ is ample iff the previous inequalities are strict.
\end{lemma}
\proof
The result is trivial for $\delta=8$.
If $\delta \le 7$, it is well known that the effective cone $\Eff(X)$ is generated by the $(-1)$-curves of $X$ (cf. \cite[Theorem 8.2.19]{D_2012}). Since the nef cone $\Nef(X)$ is the dual of $\Eff(X)$, the inequalities in the statement are obtained intersecting $\beta$ with all the $(-1)$-curves in $X$.
\endproof

We prove now some rationality results for moduli spaces of rational curves on del Pezzo surfaces. 

\begin{lemma}\label{Lemma:RedP2}
Let $X=X_\delta$ be a del Pezzo surface, $X_\delta \not\simeq \P^1\times\P^1$. Then, for any polarisation $\beta=(a; b_1, \ldots, b_{9-\delta})$ with $(\delta, \beta)\neq (1,K_{X_\delta})$, let $M_{\bf x}$ be a fibre of the following evaluation map
$$\ev_{X_{\delta}}\colon \M_{0,m}^{\bir}(X_\delta, \beta) \to (X_\delta)^m$$
over a point ${\bf x}=(x_1,\dots, x_m)$ with $x_i \in X_\delta \setminus \cup_iE_i$ for all $i$.

Consider the point
$${\bf p}:=(\underbrace{p_1, \ldots, p_1,}_\text{$b_1$-times} \ldots, \underbrace{p_{9-\delta}, \ldots, p_{9-\delta},}_\text{$b_{9-\delta}$-times} )\in(\P^2)^{b}$$
and the evaluation map 
$$\ev_{\P^2}\colon\M_{0,m+d\cdot b}(\P^2, d\cdot a) \to (\P^2)^{m+d\cdot b}.$$
Then if the fibre $\ev_{\P^2}^{-1}(\pi({\bf x}), {\bf p})$ is irreducible, it (birationally) dominates $M_{\bf x}$.
\end{lemma}

\begin{proof}
We know, by hypothesis that $\M_{0,0}^{\bir}(X_\delta, \beta)$ is irreducible (if non-empty), cf.~\cite[Theorem 5.1]{T_2009}. This implies that any irreducible component $M'$ of $M_{\bf x}$ contains a curve with smooth source. Moreover \cite[Proposition~4.14]{D_2001} implies that any irreducible component of $M_{\bf x}$ has expected dimension. We define the following rational map:
$$\gamma\colon M_{\bf x} \dashrightarrow \M_{0,m+d\cdot [b]}(\P^2, d\cdot a),$$
$$[f]\mapsto [\pi \circ f]$$
where $\M_{0,m+d\cdot [b]}(\P^2, d\cdot a)$ is the quotient of $\M_{0,m+d\cdot b}(\P^2, d\cdot a)$ by the action of the symmetric group on the $b$ points. This map is birational onto the image, which corresponds to morphisms passing through $\pi({\bf x})$ and through the centres of the blow-up with the appropriate multiplicities.
\end{proof}

\begin{theorem}\label{Thm:dPM00}
Let $X_\delta$ be a del Pezzo surface of degree $\delta$, $X_\delta\not\simeq\P^2, \P^1\times \P^1$ and let $\beta=(a; b_1, \ldots, b_{9-\delta})$ be an ample class. If $3a\ge 2b$, then the moduli space $\M_{0,0}^{\bir}(X_\delta, d \beta)$ is non-empty and unirational for all $d\ge 1$.
\end{theorem}

\begin{proof}
We know by \cite[Theorem 5.1]{T_2009} that the spaces $\M_{0,0}^{\bir}(X_\delta, \beta)$ are irreducible for all polarisation $\beta=(a; b_1, \ldots, b_{9-\delta})$ verifying the hypothesis. 
If $\delta \ge 6$, the condition $3a\ge 2b$ is verified for every polarisation, thanks to Lemma~\ref{Lem:ConesdP}.

By Lemma~\ref{Lemma:RedP2}, it is enough to study the fibre $\ev_{\P^2}^{-1}({\bf p})$, so keeping Notation~\ref{Not:GenFibre}, we consider the map
$$\Psi_{d\cdot b}\colon \M_{0,d\cdot b}(\P^2, d\cdot a) \to \M_{0,d\cdot b} \times (\P^2)^{d\cdot b}.$$
and look at the fibres over $\U_{d\cdot b}\times (\P^2)^{d\cdot b}$, where $\U_{d\cdot b}\subset \M_{0,d\cdot b}$ is the open subset of pairwise distinct points. If $\delta \ge 7$, Lemma~\ref{Lem:ConesdP} implies that $db < da+1$, so we look at the diagram 
\[\xymatrix{
\M_{0,d\cdot a +1}(\P^2, d\cdot a) \ar[r]^-{\Psi_{d\cdot a +1}} \ar[d]^{\varphi} & \M_{0,d\cdot a +1} \times (\P^2)^{d\cdot a +1}\ar[d]^{\psi\times \pi} \\
\M_{0,d\cdot b}(\P^2, d\cdot a) \ar[r]^-{\Psi_{d\cdot b}}  & \M_{0,d\cdot b} \times (\P^2)^{d\cdot b}. }\]
where $\varphi$ and $\psi$ are the forgetting maps and $\pi$ is the projection on the first $d\cdot b$ factors. Lemma~\ref{Lem:MorPn} implies that the fibres of $\Psi_{d\cdot b}$ over $\U_{d\cdot b}\times (\P^2)^{d\cdot b}$ are unirational and this proves the thesis if $\delta \ge 7$. The same argument proves the statement for smaller $\delta$, if $d(b-a)\le 1$, since in these cases the number of marked points $d\cdot b$ is smaller than $d\cdot a+1$.
\\Assume now that $\delta\le 6$ and $d(b-a)\ge 2$. We need to apply Lemma~\ref{Lemma:PnNonEmptyFibre}: the numerical condition (1) is implied by $3a\ge 2 b$. For condition (2), the inequality $3a\ge 2 b$ (or Lemma~\ref{Lem:ConesdP} for the case $\delta = 6$) implies 
$$d\cdot a + 1 \ge \begin{cases}
6 \text{ if $\delta =6$};\\
5  \text{ if $\delta \in [1,5].$}
\end{cases}$$
So, after reordering the points ${\bf p}$, we can assume that, for $\delta \le 5$ (resp. $\delta=6$) the first $4$ (resp. $6$) marked points are given by
$$\begin{pmatrix}
1 & 0 &0 &1 \\
0 &1 &0 & 1  \\
0 &0 & 1 &1
\end{pmatrix}
 \ \ \ \ \Bigg( \text{resp. } \begin{pmatrix}
1 & 0 &0 &1 & 0 &0 \\
0 &1 &0 & 0 &1 &0  \\
0 &0 & 1 &0 &0 & 1
\end{pmatrix}\Bigg)$$
So also condition $(2)$ is satisfied, Lemma~\ref{Lemma:PnNonEmptyFibre} applies and $\ev_{\P^2}^{-1}({\bf p})$ has the structure of a projective fibration over a rational variety $\M_{0,d\cdot b}$ (cf.~\cite{K_1993}).
\end{proof}

\begin{remark}
The numerical condition $3a\ge 2b$ is verified for every polarisation if $\delta \ge 6$. Decreasing the degree, this condition is proper and becomes particularly restrictive if $\delta \le 4$: indeed, for those degrees not even 
the anticanonical polarisation verifies the numerical condition.
\end{remark}

The same strategy provides rational simple connectedness for del Pezzo surfaces of high degree.

\begin{proof}[Proof of Theorem~\ref{Thm:dPRSC}]
We can assume that $X_\delta\not\simeq\P^2, \P^1\times \P^1$. Irreducibility of the moduli spaces $\M_{0,m}^{\bir}(X_\delta, d\cdot (3H-\sum_{i=1}^{9-\delta}E_i))$, with $m\ge 2$, is guaranteed by \cite[Theorem 5.1]{T_2009}. 

Lemma~\ref{Lemma:RedP2} implies that we need to prove that the fibre $\ev_{\P^2}^{-1}(\pi({\bf x}), {\bf p})$ is irreducible and rationally connected for general ${\bf x}$, so, as in Theorem~\ref{Thm:dPM00}, we look at the map
$$\Psi_{m + d(9-\delta)}\colon \M_{0,m + d(9-\delta)}(\P^2, 3d) \to \M_{0,m + d(9-\delta)} \times (\P^2)^{m + d(9-\delta)}.$$
We want to apply Corollary~\ref{Cor:SuffRSC}, so let us define, for all $d\ge 1$, the function $m_d$ as
$$m_d:=\left\lfloor \frac{(2\delta-9)d}{2}\right\rfloor +1.$$

Since $\delta\ge 5$, for $(2\delta-9)d\ge 2(m_d-1)$ and this is exactly the numerical condition (1) form Lemma~\ref{Lemma:PnNonEmptyFibre} with $m=m_d$.  After reordering the points $(\pi({\bf x}), {\bf p})$, as in Theorem~\ref{Thm:dPM00}, we gain also condition (2), so the fibre $\ev_{\P^2}^{-1}(\pi({\bf x}), {\bf p})$ is a projective fibration over $\M_{0,m_d + d(9-\delta)}$, which is rational (cf. \cite{K_1993}). Since the function $m_d$ is strictly increasing, we obtain the strong statement.
\end{proof}

If $\delta\le 4$, rational simple connectedness becomes more subtle, as the following example shows. 

\begin{example}
Let us consider $\M_{0,2}^{\bir}(X_4, -K_{X_4})$, which is irreducible by \cite[Theorem 5.1]{T_2009}. Using Lemma~\ref{Lemma:RedP2}, we deduce that the fibre over a general ${\bf x} \in (X_4)^2$ of the evaluation map
$$\ev_{X_4}\colon \M_{0,2}^{\bir}(X_4, -K_{X_4}) \to (X_4)^2$$
is dominated by the fibre $\ev_{\P^2}^{-1}(\pi({\bf x}), {\bf p})$ of
$$\ev_{\P^2}\colon\M_{0,7}(\P^2, 3) \to (\P^2)^{7}.$$
One can show that this general fibre is isomorphic to a smooth quartic in $\P^2$ (cf. \cite[Proposition 3.2.3]{T_2005}).
\end{example}


\section{The Fano threefold $V_5$}\label{Sec:V5_1}

\subsection{The birational geometry of $V_5$}

We introduce the central object of this section. For more details one can look at \cite[pag.~60-61]{IP99}, \cite[Section~2]{S_14}, \cite[Chapter~7]{CS_2016} and \cite[Section~5.1]{KPS}.

\begin{definition}
A smooth Fano threefold $X$ with $\rho(X)=1$, Fano index $\iota(X)=2$ and degree $(H^3)=5$ is denoted by $V_5$. 
\end{definition}

From Iskovskikh's classification of smooth Fano threefolds with $\rho=1$ (cf. \cite[Section~12.2]{IP99}) we know that $V_5$ is isomorphic to the linear section of the Grassmannian $\Gr(2,5) \subset \P^9$ by a general linear subspace $\P^6 \subset \P^9$.

Since $V_5$  verifies the index condition $\iota = n-1$, where $n$ is the dimension, in the literature, sometimes authors refer to $V_5$ as the \emph{del Pezzo threefold} of degree 5.

Let us look now at the $\SL_2(\C)$-action on $V_5$: one takes a vector space $U$ with $\dim U = 2$, chooses a basis of $\wedge^2U$ to identify $U$ and its dual $U^\vee$. Let us denote by $S_n:=\Sym^n(U)$ the symmetric tensor and consider the Clebsch-Gordan decomposition of $\wedge^2 S_4$ as $\SL(U)$-module:
$$\wedge^2 S_4 \simeq S_2 \oplus S_6.$$
This shows how to induce a natural $\SL(U)$-action on $V_5=\Gr(2,S_4)\cap \P S_6$, which is the intersection of two $\SL(U)$-invariant varieties.

Using the description provided in \cite[Section~3]{MU_1983}, one deduces the orbit structure $V_5$ with respect to the $\SL_2(\C)$-action (see also \cite[pag.~60-61]{IP99}).

\begin{lemma}{\cite[Lemmas~1.5-1.6]{MU_1983}}\label{Lemma:V5OrbitsDesc}
There is an action of $\SL_2(\C)$ on $V_5$ with three orbits with the following description;
\begin{itemize}
\item a $1$-dimensional orbit $\sigma$ (with representative $[x^6]\in \P S_6$) which is a rational normal sextic in $\P^6$;
\item a $2$-dimensional orbit $E\setminus \sigma$ (with representative $[x^5y]\in \P S_6$), where $E$ is a quadric surface which is the tangential scroll of $\sigma$. The normalisation $\nu\colon \P^1\times\P^1 \to E$ is determined by a (non-complete) system of degree $(1,5)$;
\item a $3$-dimensional orbit $U$ (with representative $[xy(x^4+y^4)]\in \P S_6$).
\end{itemize}
\end{lemma}

The following construction is essential for our analysis. This can be found in \cite[pag 147, Example (i)]{IP99}, \cite[Section~7.7]{CS_2016}.

\begin{lemma}\label{Lemma:ProjLine}
Let $l$ be a line in $V_5$. Then the projection $\phi_l\colon V_5 \dasharrow \P^4$ from $l$ is dominant on a quadric threefold $\Q_3$ and is birational. In particular $V_5$ is a rational threefold.

Let $D_l$ be the divisor spanned by the lines in $V_5$ meeting $l$. Then $D_l$ is a hyperplane section of $V_5 \subset \P^6$ and $\phi_l(D_l)=\gamma_l$ is a twisted cubic in $\Q_3$.
\end{lemma}

\subsection{Rational curves on $V_5$ and rational simple connectedness}\label{SubSec:V5RSC}

Let us list some remarkable properties which hold for some moduli spaces of rational curves on $V_5$. For more details, see \cite[Section~5.1]{KPS}.
\begin{itemize}
\item let $x \in V_5$ be a general point, then the VMRT is the union of three (reduced) points (cf.~\cite{I_94});
\item  more precisely, the moduli space $\M_{0,0}(V_5,1)$ is isomorphic to the Hilbert scheme of lines and (cf.~\cite[Proposition~1.6(i)]{I_79}) 
$$\HH_{0,1} \simeq \P S_2 (\simeq \P^2).$$ 
Looking at the incidence correspondence we have the following diagram:
\begin{equation}\label{V5Incidence}
\xymatrix {
  \M_{0,1}(V_5,1) \ar[r]^-{\ev_1}_-{3 :1} \ar[d]_{\varphi} & V_5  \\
  \P^2 
}.
\end{equation}
The map $\ev_1$ is a $3$-to-$1$ cover ramified on the surface $E$ and fully ramified on the sextic $\sigma$ (cf.~\cite[1.2.1(3)]{I_94}, \cite[Corollary~2.24]{S_14}). 
\end{itemize}

We study here the moduli spaces $\M_{0,m}(V_5,d)$ for $d\ge 2$ and prove some rationality results, when $m\le 2$. Quite recently, the structure of $\M_{0,0}(V_5,d)$ has been studied in \cite[Theorem~7.9]{LT_17}. In particular, its decomposition in irreducible components is deduced. Our approach is different for $V_5$: we first prove that $\M_{0,0}^{\bir}(V_5,d)$ is unirational (and irreducible) of dimension $2d$. From this the structure of $\M_{0,0}(V_5,d)$ follows.

\begin{definition}\label{Def:V5Mline}
For any $d\ge 1$ the closure in $\M_{0,0}(V_5,d)$ of the moduli space of morphisms  which factor through a line is denoted by $\M_{0,0}^{\lin}(V_5,d)$.
\end{definition}

Our main structure result is the following.

\begin{theorem}\label{Thm:V5M0}
For any $d \geq 1$ the moduli space $\M_{0,0}(V_5,d)$ has pure dimension $2d$. Moreover, for $d \ge 2$, the moduli space $\M_{0,0}(V_5,d)$ has two irreducible components: 
$$\M_{0,0}(V_5,d) = \M_{0,0}^{\bir}(V_5,d)\cup \M_{0,0}^{\lin}(V_5,d).$$ 
Furthermore,
\begin{itemize}
\item $\M_{0,0}^{\bir}(V_5,d)$ is unirational;
\item $\M_{0,0}^{\lin}(V_5,d)$ is rational.
\end{itemize}
\end{theorem} 

In order to exploit the quasi-homogeneity of $V_5$, we need few lemmas.

\begin{lemma}\label{Lemma:V5Lines}
Let $C$ be a curve in $V_5$. Then:
\begin{enumerate}
\item there exists a line $l$ with $l \cap C = \emptyset$;
\item If $C$ is reduced of degree $d$ and meets the dense orbit $U$, there exists a line $l$ with $l \cap C = \emptyset$ and such that its intersection with the divisor $D_l$ (spanned by the lines in $V_5$ meeting $l$)  is a union of $d$ reduced points.
\end{enumerate}
\end{lemma}

\begin{proof}
Looking at the incidence correspondence \eqref{V5Incidence}, we notice that he locus of lines meeting $C$ is given by $\varphi(\ev_1^{-1}(C))$, which is a curve in $\P^2$. This proves (1).

For a general line, the surface $D_l$ meets properly the $\SL_2(\C)_2$-orbits. By Kleiman-Bertini's theorem (cf. \cite[Theorem~2(ii)]{K_1974}), we get that there exists a line $l$ such that $C \cap D_l$ is a finite reduced union of points.
\end{proof}

\begin{lemma}\label{Lemma:V5Orbits}
Any irreducible component of $\M_{0,0}(V_5, d)$ contains stable curves whose image meets $U$, the dense $\SL_2(\C)$-orbit.
\end{lemma}

\begin{proof}
Since any irreducible component of $\M_{0,0}(V_5, d)$ has dimension at least $2d$ (cf. \eqref{expected_dimension}), the morphisms that factor through the orbits $E\setminus \sigma$ and $\sigma$ cannot form irreducible components. Indeed, the subscheme of stable maps that factor through the rational sextic curve $\sigma$, is isomorphic to $\M_{0,0}(\P^1,d/6)$, when $d$ is a multiple of 6 and empty otherwise. In particular is has dimension $d/3 - 2 < 2d$, when non-empty. Recall that the closure $E$ of the two-dimensional orbit is a denormalisation of $\P^1 \times \P^1$ with associated map $\nu\colon \P^1 \times \P^1 \to S$ of degree $(1,5)$, cf. Lemma~\ref{Lemma:V5OrbitsDesc}. A similar computation shows that the dimension of stable map that factor through $E$ has dimension $2d-1 < 2d$ when nonempty: when $d$ is a multiple of 6, this moduli space is given by $\M_{0,0}(\P^1\times\P^1,\frac{d}{6}(1,5))$. So any irreducible component of $\M_{0,0}(V_5, d)$ contains stable maps whose image meet the dense orbit in $V_5$.
\end{proof}

The key result is the following.

\begin{proposition}\label{Prop:V5RatComp}
For any $d \geq 1$ the moduli space $\M_{0,0}^{\bir}(V_5,d)$ is irreducible, unirational of dimension $2d$.
\end{proposition}

\begin{proof}
Since every irreducible component $M$ of $\M_{0,0}^{\bir}(V_5,d)$ has dimension $\dim M \ge 2d$ (cf.~\cite[Theorem~II.1.2]{K_1996}), let  $f\colon \P^1 \to V_5$ be a general element of the dense subset in $M$. Since $f(\P^1)\cap U\neq \emptyset$, the pull-back $f^*T_{V_5}$ of the tangent bundle of $V_5$ is globally generated (cf.~\cite[Example~4.15(2)]{D_2001}). We thus have $H^1(\P^1,f^*T_{V_5}) = 0$ and $\dim H^0(\P^1,f^*T_{V_5}) = 2d$. This proves that $\M_{0,0}^{\bir}(V_5,d)$ has dimension $2d$.

We now prove that $\M_{0,0}^{\bir}(V_5,d)$ is irreducible and unirational. Lemma~\ref{Lemma:V5Orbits} implies that there exists a map $[f_0] \in M$ such that $f_0(\P^1)$ meets the dense orbit $U$ and therefore Lemma~\ref{Lemma:V5Lines}(1) implies that there exists a line $l$ such that $l \cap f_0(\P^1) = \emptyset$. In particular, projecting from $l$ as described in Lemma~\ref{Lemma:ProjLine}, we obtain a rational map at the level of moduli spaces:
$$\Phi_l\colon M \dasharrow \M_{0,[d]}^{\bir}(\Q_3,d)$$
$$[f] \mapsto [\phi_l \circ f],$$  
where $\M_{0,[d]}^{\bir}(\Q_3,d)$ is the quotient of $\M_{0,d}^{\bir}(\Q_3,d)$ by the action of the symmetric group. Since the projection $\phi_l$ is birational, note that $\Phi_l$ is birational onto its image $\Phi_l(M)$. Let us study this image: for a general $f \in M$, its image verifies $f(\P^1) \cap l = \emptyset$, by Lemma~\ref{Lemma:V5Lines}(1) and the image $[\phi_l(f)]$ in $\M_{0,d}^{\bir}(\Q_3,d)$ is a stable map of degree $d$. Furthermore $\phi_l(f)$ meets the hyperplane section $D_l$ covered by lines meeting $l$ in $d$ distinct reduced points by Lemma~\ref{Lemma:V5Lines}(2), therefore $\Phi_l(f)(\P^1)$ meets the twisted cubic $\gamma_l$ in $d$ distinct reduced points. 

In particular, let $\ev_{d} \colon \M_{0,d}^{\bir}(\Q_3,d) \to \Q_3^d$ be the evaluation map. We have that $\ev_{d}^{-1}(\gamma_l^d)$ dominates $\Phi_l(M)$, via the quotient by the symmetric group $\mathcal{S}_d$. To prove unirationality of $M$ it is therefore enough to show that $\ev_{d}^{-1}(\gamma_l^d)$ is unirational of dimension $2d$.

Let $\varphi_{d+1}\colon \M_{0,d+1}^{\bir}(\Q_3,d) \to \M_{0,d}^{\bir}(\Q_3,d)$ the map forgetting the last marked point and  $\ev_{d+1} \colon \M_{0,d+1}^{\bir}(\Q_3,d)\to \Q_3^{d+1}$ the evaluation at all marked point. Then $\varphi_{d+1} : \ev_{d+1}^{-1}(\gamma_l^d \times \Q_3) \to \ev_{d}^{-1}(\gamma_l^d)$ is dominant and since the fibres of the last map have dimension one, it is enough to prove that $\M:=\ev_{d+1}^{-1}(\gamma_l^d \times \Q_3)$ is irreducible and unirational of dimension $2d+1$.

We prove that $\M$ is actually rational. Let $\M'$ be an irreducible component of $\M$; following Notation~\ref{Not:GenFibre}, consider the map
$$\Psi_{d+1}\colon \M' \to \M_{0,d+1}\times \gamma_l^d \times \Q_3.$$
Since $\M$ is locally defined by $2d$ equations, coming from the local equations of $\gamma_l$ in $\Q_3$, the irreducible component $\M'$ has dimension at least $2d +1$.

Proposition~\ref{Prop:GenRkQn} implies that $\Psi_{d+1}(\M')$ contains an elements whose associated matrix $A$ has rank at least $2\left\lceil \frac{d-1}{2} \right\rceil$. We study separately two cases. 

\begin{itemize}
\item If $d = 2k$ even, a general element in $\Psi_{d+1}(\M')$ defines a matrix $A$ of rank $2k$ and the fibre is a point, by Lemma~\ref{Lemma:MorQn}. The map $\Psi_{d+1}$ is therefore birational to its image and has to be dominant since $\M_{0,d+1}\times \gamma_l^d \times \Q_3$ has dimension $2d +1$. In particular $\M'$ is rational and there is a unique such irreducible component.

\item For $d = 2k+1$ odd, a general element in $\Psi_{d+1}(\M')$ defines a matrix $A$ of rank $2k$, the fibre is an open subset in $\P^1$ and even an open subset of a $\P^1$ bundle over $\Psi_{d+1}(\M')$. 
Indeed, by Lemma~\ref{Lem:PfComp}, we have a rational section given by the vector $N_0 = ((-1)^{j}\pf(A(0,j)))_j$. The image $\Psi_{d+1}(\M')$ is then given by the locus
$$\Psi_{d+1}(\M') = \{({\bf t}, {\bf x}) \in \M_{0,d+1}\times \gamma_l^d \times \Q_3 \ | \ \pf(A) = 0\}.$$
Consider the map $\theta\colon \Psi_{d+1}(\M') \to \M_{0,d+1}\times \gamma_l^d$ obtained by projection. This map is surjective and its fibre is a linear section of $\Q_3$ therefore a (rational) $2$-dimensional quadric. Since $\theta$ has a rational section (take $\Psi_{d+1}(\M') \cap (\M_{0,d+1} \times (\gamma_l)^d \times L)$ where $L$ is a general line in $\Q_3$), we see that $\M'$ is rational and there is a unique such irreducible component.
\end{itemize}
This concludes the proof.
\end{proof}

\begin{proof}[Proof of Theorem~\ref{Thm:V5M0}]
We first prove that $\M_{0,0}^{\lin}(V_5,d)$ is rational. We have the following isomorphism given by composition of stable maps:
$$\M_{0,0}(\P^1,d) \times \M_{0,0}(V_5,1) \to \M_{0,0}^{\lin}(V_5,d).$$
It is well known that the moduli space $\M_{0,0}(\P^1,d)$ is rational of dimension $2d-2$ (cf.~ \cite[Corollary~1]{KP_2001}) and $\M_{0,0}^{\lin}(V_5,1)$ is isomorphic to $\P^2$ so that $\M_{0,0}^{\lin}(V_5,d)$ is rational of dimension $2d$.

We prove by induction on $d$ that $\M_{0,0}(V_5,d)$ has two irreducible components $\M_{0,0}^{\bir}(V_5,d)$ and $\M_{0,0}^{\lin}(V_5,d)$, both of dimension $2d$. Indeed, if there exists an irreducible component not contained in $\M_{0,0}^{\bir}(V_5,d)\cup\M_{0,0}^{\lin}(V_5,d)$, then it is covered by the images of the gluing maps: 
$$\M_{0,1}(V_5,d_1) \times_{V_5} \M_{0,1}(V_5,d_2) \to \M_{0,0}(V_5,d)$$ 
with $d_1 + d_2 = d$. 
By induction assumption, this space has dimension at most $(2d_1 + 1) + (2d_2 + 1) - 3 = 2d - 1 < 2d$. Contradiction. So the theorem is a consequence of Proposition~\ref{Prop:V5RatComp}.
\end{proof}

This argument can be easily adapted to obtain rational simple connectedness.

\begin{theorem}\label{Main_V_5}
The Fano threefold $V_5$ is rationally simply connected.
\end{theorem}

\begin{proof}
Let $\ev_2 : \M_{0,2}^{\bir}(V_5,d) \to V_5^2$ be the evaluation map. We want to prove that the general fibre $M_{\bf x}=\ev_2^{-1}({\bf x})$ is rationally connected of dimension $2d - 4$. We will actually prove that $M_{\bf x}$ is unirational. 

Since $\M_{0,2}^{\bir}(V_5,d)$ is irreducible of dimension $2d + 2$, any irreducible component $M'$ of $M_{\bf x}$ contains a curve with smooth source and \cite[Proposition~4.14]{D_2001} implies that $M'$ has expected dimension $2d-4$.

Let $l$ be a line passing through $x_1$ but not though $x_2$. By composition with the projection $\phi_l\colon V_5 \dasharrow \Q_3$ (cf. Lemma~\ref{Lemma:ProjLine} for the notation), a stable map $[f] \in M_{\bf x}$ with multiplicity one in $x_1$ is sent to a stable map $\phi_l \circ f$ in $\Q_3$ of degree $d-1$ whose image meets the twisted cubic $\gamma_l$ in $d-2$ points and which passes through the fibre $\ell:=\phi_l^{-1}(x_1)$ over $x_1$. So, we obtain a rational map at the level of moduli spaces:
$$\Phi_l\colon M' \dasharrow \M_{0,[d]}^{\bir}(\Q_3,d-1)$$
$$[f] \mapsto [\phi_l \circ f],$$  
birational onto its image, Here, $\M_{0,[d]}^{\bir}(\Q_3,d-1)$ is the quotient of $\M_{0,d}^{\bir}(\Q_3,d-1)$ by the action of the symmetric group. We look at the image: by construction, $\ell$ is a line meeting $\gamma_l$. Furthermore, since $x_2$ is in general position, it is outside the indeterminacy locus of $\phi_l$, we deduce that the image $\Phi_l(M_{\bf x})$ is dominated (via the quotient by the symmetric group $\mathcal{S}_d$) by $\ev_{d}^{-1}(\gamma_l^{d-2}\times \ell \times \{ \phi_l(x_2)\})$, where $\ev_{d} : \M_{0,d}(\Q_3, d-1) \to \Q_3^{d}$ is the usual evaluation map.
\\It is therefore enough to prove that $\M:=\ev_{d}^{-1}(\gamma_l^{d-2}\times \ell \times \{ \phi_l(x_2)\})$ is irreducible and unirational of dimension $2d - 4$. 

As in Proposition~\ref{Prop:V5RatComp} we prove that $\M$ is actually rational. Let $\M'$ be an irreducible component of $\M'$; following Notation~\ref{Not:GenFibre}, we look at the map
$$\Psi_{d}\colon \M' \to \M_{0,d}\times \gamma_l^{d-2}\times L \times \{ \phi_l(x_2)\}.$$
The same argument of Proposition~\ref{Prop:V5RatComp} implies that $\M'$ has dimension at least $2d -4$.

Since the points ${\bf x}=(x_1,x_2)$ are in general position, we can apply Proposition~\ref{Prop:GenRkQn} which implies that $\Psi_{d}(\M')$ contains an elements whose associated matrix $A$ has rank at least $2\left\lceil \frac{d-2}{2} \right\rceil$. We study separately two cases. 

\begin{itemize}
\item If $d = 2k+1$ odd, a general element in $\Psi_{d}(\M')$ defines a matrix $A$ of rank $2k$ and we conclude by Lemma~\ref{Lemma:MorQn} that $\Psi_{d}$ is birational onto $\M_{0,d}\times \gamma_l^{d-2}\times L \times \{ \phi_l(x_2)\}$, which is rational.
\item For $d = 2k$ even, a general element in $\Psi_{d}(\M')$ defines a matrix $A$ of rank $2k-2$, the fibre is an open subset in $\P^1$ and even an open subset of a $\P^1$ bundle over $\Psi_{d+1}(\M')$, by Lemma~\ref{Lem:PfComp}. We conclude as in Proposition~\ref{Prop:V5RatComp}.
\end{itemize}
This concludes the proof.
\end{proof}

It would be interesting to answer the following, since our methods cannot be adapted if the number of marked points increases.

\begin{question}
Is the Fano threefold $V_5$ strongly rationally simply connected?
\end{question}

\subsection{Moduli spaces of rational curves on $V_5$ in low degree}\label{Subsec:V5_2}

In this final part we give some explicit construction for Hilbert schemes of rational curves in $V_5$, which imply rationality results for low-degree $\M_{0,0}^{\bir}(V_5,d)$.
These are well-known results, especially up to degree $5$. The case $d=6$ of sextic curves has been studied in \cite{TZ_2012b} and the authors prove that this space is rational (cf.~\cite{TZ_2012a}, \cite[Theorems~1.1-5.1]{TZ_2012b}).

For any integer $d\ge 1$, the Hilbert scheme of degree $d$ rational curves in $V_5$ is denoted by $\HH_{0,d}=\HH_{0,d}(V_5)$.

Up to cubic curves, the description of $\HH_{0,d}$ is classical.
\paragraph{\bf Lines in $V_5$} $\HH_{0,1}\simeq \P S_2$ (cf.~\cite[Proposition~1.6(i)]{I_79}, \cite[Theorem~1]{FN_1989}).
\paragraph{\bf Conics in $V_5$} $\HH_{0,2}\simeq \P S_4$ (cf.~\cite[Proposition~1.22]{I_94}).
\paragraph{\bf Cubics in $V_5$} $\HH_{0,3}\simeq \Gr(2,S_4)$ (cf.~\cite[Proposition~2.46]{S_14}).

This implies that $\M_{0,0}^{\bir}(V_5,d)$ is rational for $1\le d\le 3$.

\

\paragraph{\bf Quartics in $V_5$} Although the following results are well-known to experts, we provide here geometric proofs.

\begin{proposition}
The moduli space $\M_{0,0}^{\bir}(V_5,4)$ is rational of dimension 8.
\end{proposition}

\begin{proof}
A rational quartic $C$ in $V_5$ is clearly non-degenerate, so its linear span is a linear subspace $H_C$ isomorphic to $\p^4 \subset \p^6$. This linear subspace $H_C$ intersects $V_5$ along a degree $5$ curve and, by adjunction formula, this is an elliptic curve. Therefore $H_C \cap V_5$ is the union of $C$ and a line $l_C$ bisecant to $C$. We therefore obtain a rational map 
$$\M_{0,0}^{\bir}(V_5,4) \dasharrow \M_{0,0}^{\bir}(V_5,1).$$
Furthermore, if $l \subset V_5$ is a line and $H$ is a general codimension $2$ linear subspace in $\p^6$ with $l \subset H$, then $H \cap V_5$ is the union of $l$ and a degree $4$ rational curve. This proves that $\M_{0,0}^{\bir}(V_5,4)$ is birational to the variety $Z$ obtained via the following fibre product:
$$\xymatrix{Z \ar[r] \ar[d] & {\rm Fl}(2,5;7) \ar[d] \\
\M_{0,0}^{\bir}(V_5,1) \ar[r] & \Gr(2,7), }$$
where $\Gr(2;7)$ is the grassmannian of lines in $\p^6$ and ${\rm Fl}(2,5;7)$ is the partial flag variety of pairs $(l,H)$ with $l$ a line in $\p^6$ and $H \supset l$ a linear subspace of codimension $2$ in $\p^6$. Since the right vertical map is locally trivial in the Zariski topology, the same is true for the left vertical map. The fibres of both vertical maps are rational (isomorphic to $\Gr(3,5)$) and $\M_{0,0}^{\bir}(V_5,1)$ is rational, proving the result.
\end{proof}

We recall some notation on (projective) normality (cf.~\cite[Definition~I.4.52]{SR_1949}). 

\begin{definition}
Let $X$ be an integral variety and let $X\subset \P^n$ be a non-degenerate embedding. One says that $X$ is a \emph{normal subvariety} in $\P^n$ if it is not a projection of a subvariety of the same degree in $\P^N$, with $N>n$.
\\One says that $X$ is \emph{linearly normal} if the restriction map
$$H^0(\P^n, \O_{\P^n}(1))\to H^0(X, \O_X(1))$$
is surjective.
\end{definition}

\begin{remark}
Let $X\subset \P^n$ be non-degenerate. Then one can show the following implications (cf.~\cite[Proposition~8.1.5]{D_2012}):
\begin{itemize}
\item $X$ is a normal subvariety in $\P^n \Rightarrow X$ is linearly normal;
\item $X$ is linearly normal and normal $\Rightarrow X$ is a normal subvariety in $\P^n$.
\end{itemize}
\end{remark}

\paragraph{\bf Bisecants in $V_5$}In order to study moduli spaces of quintic and sextic curves, we need to look at bisecant lines. Let $C\subset V_5$ be a smooth connected curve of degree $d$ and genus $g$. The set of bisecants to $C$ (i.e.~the lines in $V_5$ meeting $C$ in two points) is expected to be of codimension two in $\P S_2$, that is, finite.

\begin{lemma}\label{Lemma:NumBisec}
Keep the notation as above. \\Then the number of bisecants to $C$ is 
${{d-2}\choose{2}} -3g$.
\end{lemma}

\begin{proof}
We compute the genus $g'$ of the inverse image $C'$ of C in the incidence correspondence $\M_{0,1}(V_5,1)$ (cf.~\eqref{V5Incidence}) and the degree $\delta$ of the projection $C'\to \P S_2$. The projection $\M_{0,1}(V_5,1)\to V_5$ is 3-to-1 and ramifies exactly on $E$ (cf.~the notation of Lemma~\ref{Lemma:V5OrbitsDesc} and \cite[Lemma~2.3]{FN_1989}), which is a section of $\O_{V_5}(2)$, Hurwitz formula implies 
$$2g'-2=3(2g-2)+2d,$$
giving $g'=3g+d-2$. Moreover, $\delta=d$, since the image in $V_5$ of the inverse image in $\M_{0,1}(V_5,1)$ of a line $\ell$ in $\P S_2$ is the hyperplane section $H_l\cap V_5$ swept out by the lines which intersect the line $l$, where $l$ is the line of $V_5$ corresponding in $\P S_2$ to the orthogonal of $\ell$. The expected number is the double locus of the projection $C'\to \P S_2$ which is ${{d-1}\choose{2}} -g'$.
\end{proof}

\paragraph{\bf Quintics in $V_5$ via Desargues configurations} For any smooth rational quintic $C$ in $V_5$, one sees it is linearly normal, i.e. it is contained in a unique hyperplane $H_C \subset \P S_6$. Let us consider the surface $S=S_C:=H_C\cap V_5$. The surface $S$ is either non-normal, or a (possibly singular) del Pezzo surface of degree 5. 
\\If $S$ is non-normal, then $S= V_5 \cap H_l$,  where $l$ is a line in $V_5$ and $H_l$ is the hyperplane containing the first infinitesimal neighbourhood $\Spec(\O_{V_5}/I^2_{l/{V_5}})$ of $l$ in $V_5$, since the non-normal locus is a line (cf.~\cite[Proposition~5.8]{BS_2007}). In this case, $S$ is the union of lines in $V_5$ meeting $l$ (cf.~\cite[Proposition~2.1]{PS_1988}, \cite[Corollary~1.3]{FN_1989}) and contains a 6-dimensional family of rational normal quintics (excluding the particular case of hyperplane sections of $S$). For such a quintic, the line $l$ is a triple point of the scheme (of length 3) of bisecant lines to $C$ (cf.~Lemma~\ref{Lemma:NumBisec}). Varying the line $l$, we obtain a 8-dimensional family of quintic curves, which is smaller than the expected dimension $10$.

Let us study now the normal case. Let $\tilde S$ be the anticanonical model of $S$, then the class $\O_{\tilde S}(C)$ of $C$ in $\tilde S$ is given by $\alpha-K_{\tilde S}$, for some root $\alpha \in \Pic(\tilde S)$ (cf.~\cite[Section~8.2.3]{D_2012}). We recall that the roots of $\Pic(\tilde S)$, i.e.~ the orthogonal lattice to $K_{\tilde S}$ in $\Pic(\tilde S)$ (of
type $A_4$) are the differences $\pi_i-\pi_j$, with $i \neq j$, where $\pi_i$ are the markings of the del Pezzo surface.

One can see the roots in $\Pic(\tilde S)$ via a self-conjugate (with respect to the fundamental conic) {\it Desargues configuration} in $\P S_2$. This it a very geometric interpretation of the root system $A_4$. Assume that $S$ is smooth (i.e.~it contains ten lines, which are points in $\P S_2$). The lines of $S$ are the sums $K_S + \pi_i+\pi_j$, indexed by the 2-subsets $\{i,j\} \subset [1,5]$ and we denote them by $\ell_{ij}$. Let $M_{ij}$ be the corresponding point in $\P S_2$. Moreover, two lines $\ell_{ij}$ and $\ell_{kl}$ meet if and only if $M_{ij}$ and $M_{kl}$ are conjugate (with respect to the fundamental conic) in $\P S_2$. This is the case if and only if $\{i,j\} \cap \{k,l\}=\emptyset$.
So the ten points $M_{ij}$ and their polar lines $M_{ij}^\perp$ form a Desargues configuration, i.e.~\emph{a line contains three points and a point lies in three lines}.
This configuration contains twenty triangles indexed by the roots of $\Pic(\tilde S)$: a root is an ordered pair $(i,j)$ and we associate to it a triangle in the following way. Take for instance $(1,2)$: we associate to it the triangle $M_{13}M_{14}M_{15}$, which is in perspective from $M_{12}$ to its conjugate $M_{23}M_{24}M_{25}$.

\begin{proposition}
The moduli space $\M_{0,0}^{\bir}(V_5,5)$ is rational of dimension 10.
\end{proposition}

\begin{proof}
Let $C$ be a smooth rational quintic curve in $V_5$. We can assume it is contained a normal hyperplane section $S=S_C$ of $V_5$, in fact we may assume the smoothness of $S$, since the map $C\mapsto S$ is a $\P^4$-fibration: this can be seen computing $H^0(\tilde S,\alpha-K_{\tilde S})=5$ on the canonical model of $S$.
\\We know that the class $\O_{\tilde S}(C)$ in $\Pic(S)$ is $\alpha-K_{S}$ for some root $\alpha$ which we choose labeled as $(1,2)$. We obtain:
$$(C\cdot l_{ij}) =
\begin{cases}
0 \text{ if } i=1, j=3,4,5;\\
2 \text{ if } i = 2, j = 3,4,5;\\
1 \text{ otherwise}. 
\end{cases}
$$
the bisecants to $C$ form the triangle $M_{23}M_{24}M_{25}$ (in the Desargues configuration explained above) associated to the root $(2,1)$. Furthermore, we remark that the triangles correspond exactly to the sets of three lines on $S$ generating a hyperplane. We obtained this way a rational map associating to a rational quintic its set of bisecant lines:
$$\Psi=\Psi_5\colon \M_{0,0}^{\bir}(V_5,5) \dasharrow \HH$$
$$C\mapsto [M_{23},M_{24},M_{25}]$$
where $\mathcal{H} := \Hilb(3, \P S_2)$.
\\We study the fibres of $\Psi$: choose a general point of $\mathcal{H} := \Hilb(3, \P S_2)$: the corresponding set $\tau$ of three lines in $V_5$ generates a hyperplane $H$. Consider the smooth surface $\mathcal{S}:=V_5\cap H$. There exists exactly one root $\alpha\in \Pic(S)$ whose intersection product with the lines corresponding to $\tau$ is 1. The quintic rational curves in $V_5$ whose set of bisecants is $\tau$ are exactly the sections of the invertible sheaf $\alpha-K_{\tilde S}$ on S and they form a 4-dimensional projective space. Therefore the map $\Psi$, is (birationally) a $\P^4$-fibration.
\\We remark that the construction is still meaningful over the field of rational functions of $\HH$, i.e.~$\Psi$ is Zariski-locally trivial: in fact, the generic point $\eta$ of $\HH$ gives a divisor of a del Pezzo surface $S_\eta$ defined over $K=\C(\HH)$, with three lines defined over $K$. So also the root is defined over $K$, i.e. $\alpha \in \Pic(S_\eta)$.
\\To conclude, the variety $\HH$ is rational: it is easy to see that $\Hilb(3, \P^1\times \P^1)$ is rational since, once we fix a ruling $\P^1 \times \P^1\to \P^1$, there is a birational correspondence between the conics on $\P^1 \times \P^1$ and the set of three rules.
\end{proof}

\paragraph{\bf Sextics in $V_5$ the Segre nodal cubic} 

We conclude studying sextic curves in $V_5$.
 
\begin{remark}\label{Rem:Spin5}
First, we recall the classical and sporadic isomorphism $\Spin_5 \sim \Sp_4$. The group $G = \Sp_4$ has two sets of maximal parabolic subgroups. One of those can be identified with the quadric $\Q_3$ (the closed orbit of the standard representation of $G$ viewed as $\Spin_5$); the other one is the projective space $\P^3$ (the closed orbit of the standard representation of $G$ viewed as $\Sp_4$). 
\\The variety $\mathcal{B}$ of Borel subgroups of $G$, viewed inside $\Q_3 \times \P^3$, is the incidence correspondence (line, point), when we view $\Q_3$ as the set of isotropic lines in $\P^3$ with respect to the standard symplectic form.
More precisely, the fibre in $\mathcal{B}$ of a point $P\in \P^3$ is a line $\ell_p \subset \Q_3$ and the fibre in  $\mathcal{B}$ of a point $Q\in \Q_3$ is an isotropic line $l_Q\subset \p^3$. The  relations $Q\in \ell_P$ and $P\in l_Q$ are equivalent.
\end{remark}

We provide here a new geometric proof of the following result by Takagi and Zucconi (cf.~\cite[Theorem~5.1]{TZ_2012b}).

\begin{proposition} The moduli space $\M_{0,0}^{\bir}(V_5,6)$ is rational of dimension 12. More precisely, the rational map
$$\Psi_6\colon \M_{0,0}^{\bir}(V_5,6) \dasharrow \Hilb(6, \P S_2)$$
sending a general rational sextic curve to its set of bisecants is birational.
\end{proposition}

\begin{proof}
Lemma~\ref{Lemma:NumBisec} implies that a rational sextic in $V_5$ has six bisecant lines, so it remains to prove that for 6 lines in $V_5$ in general position $l_i$, with $i\in [0,5]$, there exists a unique rational sextic which is bisecant to them. Let us fix one of these lines $l_0$ and consider the projection from it, which we saw in Lemma~\ref{Lemma:ProjLine}:
$$\phi_{l_0}\colon V_5 \dasharrow \Q_3$$
which is birational. The rational sextics in $V_5$ which are bisecant to $l_0$ are in 1-to-1 correspondence with the rational quartics in $\Q_3$ which are bisecant to $\gamma_{l_0}$ (keeping the notation of Lemma~\ref{Lemma:ProjLine}). Moreover the strict transforms of the lines $l_1,\ldots l_5$ in $Q_3$ are lines $l_i' \subset \Q_3$ which are secant to $\gamma_{l_0}$. Our problem is then reduced to counting rational quartics in $\Q_3$ which are bisecant to the $l_i'$'s and to $\gamma_{l_0}$.

The  correspondence in Remark~\ref{Rem:Spin5} provides a birational correspondence between the variety parametrising the smooth non-degenerate rational quartics in $\Q_3$, and the variety parametrising the twisted cubics in $\P^3$. Indeed, a rational quartic curve $\delta \in \Q_3$ is the ruling of a quartic scroll $\Sigma_\delta \subset \P^3$ whose double locus is the corresponding twisted cubic $\gamma \subset \P^3$. Conversely, the datum of $\gamma$ allows to reconstruct $\delta$ as the locus of isotropic lines in $\P^3$ which are bisecant to $\gamma$.

Moreover, the marked twisted cubic $\gamma_{l_0}$ corresponds to the ruling of a cubic scroll $\Sigma_{\gamma_{l_0}}$. Since each line $\ell_i\in \Q_3$, with $i \in [1,5]$ meets $\gamma_{l_0}$, the cubic scroll $\Sigma_{\gamma_{l_0}}$ contains the points $P_i$ defined by $\ell_i:=\ell_{P_i}$. So our problem is now reduced to enumerating the twisted cubics $\gamma\in \P^3$ through the points $P_i\in \P^3$ and bisecant to two rules $\Sigma_{\gamma_{l_0}}$.

We perform now another birational transformation: let $\mathcal{P} \subset \P^3$ be the union of the five points $P_i$'s and consider the linear system of quadrics passing through $\mathcal{P}$. This defines a birational map
$$\varphi_{\mathcal{P}}\colon \P^3 \dasharrow \mathbb{S}_3\subset \P^4,$$
where $\mathbb{S}_3$ is the 10-nodal Segre cubic primal (cf.~\cite[Section~2]{D_2016}) and the indeterminacy locus of $\varphi_{\mathcal{P}}$ coincides with $\mathcal{P}$. This map can be seen as the composition 
$$\varphi_{\mathcal{P}}=\pr \circ v_2\colon \P^3 \hookrightarrow v_2(\P^3) \dasharrow \mathbb{S}_3\subset \P^4,$$
where $v_2\colon \P^3 \to \P^9$ is the second Veronese map and $\pr=\pr_{\langle v_2(\mathcal{P})\rangle}$ is the projection from the 4-dimensional linear space spanned by $v_2(\mathcal{P})$.

Recall that the system of twisted cubics through $\mathcal{P}$ is transformed into one of the six two-dimensional systems of lines in $\mathbb{S}_3$ (the other five are the transforms of the systems of lines through the $p_i$'s).

The image $v_2(\Sigma_{\gamma_{l_0}})$ is a conic bundle over $\gamma_{l_0}\simeq \P^1$. The normal model of $\Sigma_{\gamma_{l_0}}$ is $\P(\O_{\P^1}(1)\oplus \O_{\P^1}(2))$, while the scroll given by union of the planes spanned by the conics (the fibres) of $v_2(\Sigma_{\gamma_{l_0}})$ is $V:=\P(\O_{\P^1}(2)\oplus\O_{\P^1}(3)\oplus\O_{\P^1}(4))$. One sees that $\deg V=9$ and $V\cap \langle v_2(\mathcal{P})\rangle =v_2(\mathcal{P})$. Its projection $W:=\pr(V)$ is therefore of degree 4 and so is its dual $W^\vee$, a ruled surface in $(\P^4)^\vee$. By the formula in \cite[Chapter~IV, Example~2, p.~174]{B_1960}, this surface has one singular point $T$ in which two rules meet. The corresponding hyperplane $H$ of $\P^4$ contains two planes on $W$ meeting along a line $l_T$ and containing two conics on the transformed surface $\Sigma_{\gamma_{l_0}}$. The line $l_T$ is the transform of the requested twisted cubic $\gamma$.\end{proof}


\bibliographystyle{alpha}
\bibliography{biblio}

\end{document}